\documentclass[11pt]{article}
\usepackage[utf8]{inputenc}

\usepackage{epsfig,epsf,fancybox}
\usepackage{amsmath}
\usepackage{mathrsfs}
\usepackage{amssymb}
\usepackage{graphicx}
\usepackage{color}
\usepackage{multirow}
\usepackage{paralist}
\usepackage{verbatim}
\usepackage{galois}
\usepackage{algorithm}
\usepackage{algorithmic}
\usepackage{boxedminipage}
\usepackage{booktabs}
\usepackage{accents}
\usepackage{stmaryrd}
\usepackage{subfig}
\usepackage{appendix}
\usepackage{epstopdf}
\usepackage{amsthm}
\usepackage[top=1in,bottom=1.2in,left=1in,right=1in,xetex]{geometry}

\newtheorem{theorem}{Theorem}[section]

\newtheorem{lemma}[theorem]{Lemma}
\newtheorem{proposition}[theorem]{Proposition}

\newtheorem{definition}[theorem]{Definition}

\newtheorem{remark}[theorem]{Remark}

\newtheorem{problem}{Problem}[section]

\newcommand{\f}{\mathscr{F}}
\newcommand{\lr}{\mathcal{L}}
\newcommand{\sr}{\mathcal{S}}
\newcommand{\e}{\mathbb{E}}
\newcommand{\br}{\mathbb{R}}
\newcommand{\pr}{\mathcal{P}}
\newcommand{\s}{\mathbb{S}}
\newcommand{\ii}{\mathcal{I}}
\newcommand{\iii}{\mathbb{I}}

\newcommand{\argmin}{\mathop{\rm argmin}}

\newcommand{\de}{\Delta}

\definecolor{DSgray}{cmyk}{0,1,0,0}

\allowdisplaybreaks[4]

\title{Mean Field Games with Major and Minor Agents: the Limiting Problem and Nash Equilibrium}
\author{Ziyu Huang\footnotemark[1] \and Shanjian Tang\footnotemark[2]} 
\date{}
\begin{document}

\maketitle

\renewcommand{\thefootnote}{\fnsymbol{footnote}}

\footnotetext[1]{School of Mathematical Sciences, Fudan University, Shanghai 200433, China. E-mail: zyhuang19@fudan.edu.cn.}
\footnotetext[2]{Department of Finance and Control Sciences, School of Mathematical Sciences, Fudan University, Shanghai 200433, China. E-mail: sjtang@fudan.edu.cn. Research partially supported by
	National Natural Science Foundation of China (Grants No. 11631004 and No. 12031009).}

\begin{abstract}
In this paper, we consider  a mean field game (MFG) with a major and $N$ minor agents. We first consider the limiting problem and allow the coefficients to vary with the conditional distribution in a nonlinear way. We use the stochastic maximum principle to transform the limiting control problem into a system of two coupled conditional distribution dependent forward–backward stochastic differential equations (FBSDEs), and prove the existence and uniqueness result of the FBSDEs when the dependence between major agent and minor agents is sufficiently weak. We then use the solution of the limiting problem to construct an $\mathcal{O}(N^{-\frac{1}{2}})$-Nash equilibrium for the MFG with a major and $N$ minor agents. 


\noindent {\bf 2000 MR Subject Classification } 93E20, 60H30, 60H10, 49N70, 49J99
\end{abstract}

\section{Introduction}
Let $(\Omega,\mathbb{P})$ denote a complete probability space on which $(N+1)$ one-dimensional independent Brownian motions $\{W^i,\ W^i=(W^i_t)_{t\in [0,T]},\ 0\le i\le N\}$ are defined. In this paper, we consider mean field games (MFGs) with one major agent and $N$ minor agents. The major agent regulates his/her own state process $X_t^0$ governed by
\begin{equation}\label{pr_1}
	X_t^0=\xi^0+\int_0^t b^0\left(s,X_s^0,u_s^0,m^N_s\right)ds+\int_0^t\sigma^0\left(s,X_s^0,u_s^0,m^N_s\right)dW_s^0, \quad t\in[0,T],
\end{equation}
via the control process $u^0\in\lr_{\f^0}^2(0,T)$, and the $i$-th minor agent regulates his/her own state process $X_t^i$ governed by
\begin{equation}\label{pr_2}
	\begin{split}
		X_t^i=\xi^i &+\int_0^t b\left(s,X_s^i,u_s^i,m^N_s\right)ds+\int_0^t\sigma\left(s,X_s^i,u_s^i,m^N_s\right)dW_s^i\\
		&+\int_0^t\tilde{\sigma}\left(s,X_s^i,u_s^i,m^N_s\right)d{W}^0_s, \quad t\in[0,T],
	\end{split}
\end{equation}
via the control process $u^i\in \lr_{\f^i}^2(0,T)$, where $\{\xi^i,\ 0\le i\le N\}$ are square integrable random variables and $m_t^N$ is the empirical distribution of $\{X_t^i,1\le i\le N\}$, i.e.,
\begin{equation*}
	m_t^N(dx)=\frac{1}{N}\sum_{i=1}^N\delta_{X_t^i}(dx),\quad t\in[0,T].
\end{equation*}
Here, $\f^0$ is the natural filtration of $(\xi^0,{W}^0)$ and $\f^i$ is the natural filtration of $(\xi^0,\xi^i,W^0,W^i)$ for $1\le i\le N$ augmented by all the $\mathbb{P}$-null sets. The major agent minimizes his/her expected total cost given by
\begin{equation}\label{pr_3}
	J_0\left(u^0|u^{-0}\right):=\mathbb{E}\left[\int_0^T f^0\left(t,X_t^0,u_t^0,m_t^N\right)dt+g^0\left(X^0_T,m_T^N\right)\right],
\end{equation}
and the $i$-th agent minimizes his/her expected total cost given by
\begin{equation}\label{pr_4}
	J_i\left(u^i|u^{-i}\right):=\mathbb{E}\left[\int_0^T f\left(t,X_t^i,u_t^i,m_t^N,X_t^0\right)dt+g\left(X^i_T,m_T^N,X_T^0\right)\right],
\end{equation}
where $u^{-i}$ denotes a strategy profile of other agents excluding the $i$-th agent. We are seeking a type of equilibrium solution widely used in game theory setting called Nash equilibrium. Searching for a Nash equilibrium for an $(N+1)$-player game is known to be intractable when $N$ is large due to the high dimensionality. However, in view of the theory of McKean-Vlasov limits and propagation of chaos for uncontrolled weakly interacting particle systems \cite{AS}, it is expected to have a convergence for $(N+1)$-player Nash equilibria. The limiting system is more tractable and we can use its solution to approximate the Nash equilibrium of the finite-player games. 

MFGs with major and minor agents were first introduced by Huang \cite{MH2}. He considered a linear-quadratic infinite-horizon model in which the influence of major agent will not fade away when the number of minor agents tends to infinity. Nourian and Caines \cite{NM} generalized this model to the nonlinear case using dynamic programming  principle. Carmona and Zhu \cite{CRXZ} developed a systematic scheme to find approximate Nash equilibria for the finite-player games using a probabilistic approach. Hu \cite{HFC} studied stochastic LQ games involving a major and a large number of minor agents. We also refer to \cite{PA,MH1,VN} for construction of $\epsilon$-Nash equilibria for $N$-player games and \cite{SA,SAWR,common_3,AB2,common_1,common_2,GW,HT} for MFGs with common noise. There are two major issues of interest in MFGs with major and minor agents: (I) solvability of the limiting problem of MFG with major and minor agents, (II) convergence result for approximate Nash equilibria for $(N+1)$-player game as $N\to\infty$. The goal of this paper is to answer both questions (I) and (II) for our model.

First we discuss issue (I). Nourian and Caines \cite{NM} studied MFGs with major and minor agents assuming that the diffusion of the major agent's system is independent of the states of the major and $N$ minor agents, and the noise of the major agent does not appear in the state dynamic of minor agents. Carmona and Zhu \cite{CRXZ} used the stochastic maximum principle to reduce the solution of the MFG with major and minor agents to a forward-backward system of stochastic differential equations (FBSDEs) of the conditional McKean-Vlasov type. They proved the solvability of the linear quadratic setting, in which the diffusions of all the agents are assumed to be constants. In this paper, we allow the state coefficients to vary with the conditional distribution in a nonlinear way and allow the noise of the major agent to appear in the state dynamics of all minor agents. We allow the diffusion of an agent's state dynamic to depend on its own state, control, and the conditional distribution of all minor agents' state. We use the stochastic maximum principle to transform the limiting control problem into a system of two coupled conditional distribution dependent FBSDEs (see Theorems~\ref{thm:smp} and \ref{thm:1}). One of our main result is the existence and uniqueness of solutions of the FBSDEs when the dependence between major agent and minor agents is sufficiently weak, and the convexity parameter of the running cost of minor agents on the control is sufficiently large or the dependence of minor agents' state coefficients on the conditional distribution is sufficiently weak (see Theorem~\ref{main2_thm}). A weak monotonicity property is required for minor agents' cost functions. The proof of the existence and uniqueness of solutions of the FBSDEs is based on a continuation method in coefficients, which was developed by Hu and Peng \cite{YH2}. 

Next we discuss issue (II). Nourian and Caines \cite{NM} constructed a $\mathcal{O}(N^{-\frac{1}{2}})$-Nash equilibrium for a MFG with a major and $N$ minor agents for the case that controls of minor agents are adapted to the natural filtration of the major agent. Carmona and Zhu \cite{CRXZ} proved that the optimal control of the McKean-Vlasov dynamics forms an $\mathcal{O}(N^{\frac{-1}{d+4}})$-Nash equilibrium using results of propagation of chaos, where $d$ is the dimension of state space. In this paper, we show that the solution of the limit problem can provide an $\mathcal{O}(N^{-\frac{1}{2}})$-Nash equilibrium for MFG with a major and $N$ minor agents where control of the $i$-th agent is adapted to the natural filtration generated by his/her own noise and major agent's noise (see Theorem~\ref{thm:nash}).

The paper is organized as follows. In Section~\ref{problem_formulation}, we state our problems, assumptions and main results. In Section~\ref{pro1}, we use the stochastic maximum principle to transform the limiting control problem into a system of two coupled conditional distribution dependent FBSDEs, and give the existence and uniqueness result of the FBSDEs. In Section~\ref{sec_nash}, we use the solution of the limit problem to construct an $\mathcal{O}(N^{-\frac{1}{2}})$-Nash equilibrium for MFG with a major and $N$ minor agents. 

\subsection{Notations}
Let $(\Omega,\f,\mathbb{P})$ denote a complete filtered probability space and $\f^0$ be a sub-filtration of $\f$. We denote by $\lr(\cdot|\f^0_t)$ the law conditioned at $\f^0_t$ for $t\in[0,T]$. Let $\lr^2_{\f_t}$ denote the set of all $\f_t$-measurable square-integrable random variables. Let $\lr^2_{\f}(0,T)$ denote the set of all $\f_t$-progressively-measurable processes $\alpha$ such that $\e\big[\int_0^T |\alpha_t|^2dt\big]<+\infty$. Let $\mathcal{S}^2_{\f}(0,T)$ denote the set of all $\f_t$-progressively-measurable processes $\alpha$ such that $\e[\sup_{0\le t\le T} |\alpha_t|^2]<+\infty$. Let $\pr(\br)$ denote the space of all Borel probability measures on $\br$, and $\pr_2(\br)$ the space of all probability measures $m\in\pr(\br)$ such that 
\begin{equation*}
	|m|_{2}:=\sqrt{\int_\br x^2dm(x)}<+\infty.
\end{equation*}
The Wasserstein distance on $\pr_2(\br)$ is defined by
\begin{equation*}
	W_2(m_1,m_2)=\sqrt{\inf_{\gamma\in\Gamma(m_1,m_2)}\int_{\br^2}|x(\omega_1)-x(\omega_2)|^2 d\gamma(\omega_1,\omega_2)},\quad m_1,m_2\in\pr_2(\br),
\end{equation*}
where $\Gamma(m_1,m_2)$ denotes the collection of all probability measures on $\br^2$ with marginals $m_1$ and $m_2$. The space $(\pr_2(\br),W_2)$ is a complete separable metric space.  Let $\pr_{\f^0}^2(0,T)$ denote the set of all $\f^0_t$-progressively-measurable stochastic flow of probability measures $m=\{m_t,0\le t\le T\}$ such that $\e[\sup_{0\le t\le T} |m_t|_2^2]<+\infty$.

\section{Problems, assumptions and main results}\label{problem_formulation}
Let $T>0$ be a fixed terminal time, and 
\begin{align*}
	&b^0,\sigma^0,b,\sigma,\tilde{\sigma},f^0:[0,T]\times\br\times\br\times\pr(\br)\to\br,\\ 
	&f:[0,T]\times\br\times\br\times\pr(\br)\times\br\to\br,\\ &g^0:\br\times\pr(\br)\to\br,\qquad g:\br\times\pr(\br)\times\br\to\br.
\end{align*}
We assume that $\{\xi^i,\ 0\le i\le N\}$ are square integrable, independent, identically distributed and independent of all Brownian motions. We first give the definition of Nash equilibrium for stochastic dynamic game \eqref{pr_1}-\eqref{pr_4}.
\begin{definition}
	We call a set of admissible strategies $\{u^i,\ 0\le i\le N\}$ a Nash equilibrium for $(N+1)$ agents, if  for any $0\le i\le N$, $u^i$ is optimal for the $i$-th agent given the other agents's strategies $\{u^j,j\neq i\}$. In other words,
	\begin{equation*}
		J_i\left(u^i|u^{-i}\right)=\min_{u\in\lr_{\f^i}^2(0,T)}J_i\left(u|u^{-i}\right), \quad 0\le i\le N.
	\end{equation*}
	Given $\epsilon>0$, we call a set of admissible strategies $\{u^i,\ 0\le i\le N\}$ an $\epsilon$-Nash equilibrium, if 
	\begin{equation*}
		J_i\left(u_i|{u}_{-i}\right)-\epsilon\le\inf_{u\in\lr_{\f^i}^2(0,T)}J_i\left(u|{u}^{-i}\right)\quad 0\le i\le N.
	\end{equation*}
\end{definition}

One of the main problems of our paper is stated as follows.

\begin{problem}\label{def1}
	Find an $\epsilon$-Nash equilibrium for MFGs with a major and $N$ minor agents.
\end{problem}

Searching for a Nash equilibrium for an $(N+1)$-player game is intractable when $N$ is very large due to the high dimensionality. Therefore, we consider the limiting problem first by taking the limit when $N$ goes to infinity. We assume that each minor agent adopts the same strategy. Therefore, minor agents' distribution can be represented by a law of a single representative minor agent conditional on $\f^0$. Let ${W}=\{W_t,0\le t\le T\}$ be a one-dimensional Brownian motion and $\xi$ be a square integrable random variable, which are both independent of $(W^0,\xi^0)$. Let $\f$ be the natural filtration of $(\xi^0,\xi,W^0,W)$ augmented by all $\mathbb{P}$-null sets. The limiting problem of MFGs with major and minor agents is stated as follows.

\begin{problem}\label{def2}
	Find a solution $(\hat{u}^0,\hat{u})\in \lr_{\f^0}^2(0,T)\times\lr^2_\f(0,T)$ for the stochastic control problem
	\begin{equation*}
		\left\{
		\begin{aligned}
			&\hat{u}^0\in\argmin_{u^0\in\lr^2_{\f^0}(0,T)}J_0\left(u^0|m\right):=\e\left[\int_0^T f^0\left(t,X^{0,u^0}_t,u^0_t,m_t\right)dt+g^0\left(X^{0,u^0}_T,m_T\right)\right],\\
			&\hat{u}\in\argmin_{u\in\lr^2_{\f}(0,T)}J\left(u|\hat{u}^0,m\right):=\e\left[\int_0^T f\left(t,X^u_t,u_t,m_t,X_t^{0,\hat{u}^0}\right)dt+g\left(X^u_T,m_T,X_T^{0,\hat{u}^0}\right)\right];\\
			&X^{0,u^0}_t=\xi^0+\int_0^t b^0\left(s,X^{0,u^0}_s,u^0_s,m_s\right)ds+\int_0^t\sigma^0\left(s,X^{0,u^0}_s,u^0_s,m_s\right)dW^0_s,\\
			&X^u_t=\xi+\int_0^t b\left(s,X^u_s,u_s,m_s\right)ds+\int_0^t\sigma\left(s,X^u_s,u_s,m_s\right)dW_s\\
			&\qquad\quad\ +\int_0^t\tilde{\sigma}\left(s,X^u_s,u_s,m_s\right)dW^0_s,\quad  t\in[0,T];\\
			&m_t=\lr\left(X^{\hat{u}}_t|{\mathscr{F}}^0_t\right),\quad  t\in[0,T]. 
		\end{aligned}
		\right.
	\end{equation*}
\end{problem}

To solve Problem~\ref{def2}, we need the following assumptions. For notational convenience, we use the same constant $L\geq 1$ for all the conditions below, and assume that $\max\{L_m,l_m,l_{x^0}\}\le L$.

\textbf{(H1)} The functions $(b^0,\sigma^0,b,\sigma,\tilde{\sigma})$ are linear in  $x$ and $u$. That is, 
\begin{align*}
	\phi(t,x,u,m)=\phi_0(t,m)+\phi_1(t)x+\phi_2(t)u,\quad \phi=b^0,\sigma^0,b,\sigma,\tilde{\sigma},\quad \phi_i=b^0_i,\sigma_i^0,b_i,\sigma_i,\tilde{\sigma}_i,
\end{align*}
with $\phi_0$ growing linearly in $m$ and $(\phi_1,\phi_2)$ being bounded by $L$. The functions $(b_0,\sigma_0,\tilde{\sigma}_0)$ are $L_m$-Lipschitz continuous in $m$, and the functions $(b_0^0,\sigma_0^0)$ are $l_m$-Lipschitz continuous in $m$,

\textbf{(H2)} The functions $(f^0,g^0)(t,0,0,m)$ satisfy a quadratic growth condition in $m$. The functions $(f^0,g^0)$ are differentiable in $(x^0,u^0)$, with the derivatives growing linearly in $(x^0,u^0,m)$, being $L$-Lipschitz continuous in $(x^0,u^0)$ and being $l_m$-Lipschitz continuous in $m$. The functions $(f,g)(t,0,0,m,x^0)$ satisfy a quadratic growth condition in $(m,x^0)$. The function $f$ is of the form
\begin{align*}
	f\left(t,x,u,m,x^0\right)=f^1\left(t,x,u,x^0\right)+f^2\left(t,x,m,x^0\right)
\end{align*}
for $(t,x,u,m,x^0)\in [0,T]\times\br\times\br\times\pr_2(\br)\times\br$. The functions $(f^1,f^2,g)$ are differentiable in $(x,u)$, with the derivatives growing linearly in $(x,u,m,x^0)$, being $L$-Lipschitz continuous in $(x,u,m)$ and being $l_{x^0}$-Lipschitz continuous in $x^0$.

\textbf{(H3)} The function $g^0$ is convex in $x^0$ and the function $f^0$ is jointly convex in $(x^0,u^0)$ with a strict convexity in $u^0$, in such a way that, for some $C_{f^0}\geq L^{-1}$, 
\begin{equation*}
	\begin{split}
		&f^0\left(t,x_2^{0},u_2^{0},m\right)-f^0\left(t,x_1^0,u_1^0,m\right)-\left(f^0_{x^0},f^0_{u^0}\right)\left(t,x_1^{0},u_1^0,m\right)\cdot\left(x_2^{0}-x_1^0,u_2^{0}-u_1^0\right)\\
		\geq\  &C_{f^0}\left|u_2^{0}-u_1^0\right|^2,\qquad (t,m)\in[0,T]\times\pr_2(\br),\quad x_1^{0},x_2^{0},u_1^{0},u_2^{0}\in\br.
	\end{split}
\end{equation*}
The functions $(f^2,g)$ are convex in $x$ and the function $f_1$ is jointly convex in $(x,u)$ with a strict convexity in $u$, in such a way that, for some $C_{f}\geq L^{-1}$, 
\begin{equation*}
	\begin{split}
		&f^1\left(t,x_2,u_2,x^0\right)-f^1\left(t,x_1,u_1,x^0\right)-\left(f^1_{x},f^1_{u}\right)\left(t,x_1,u_1,x^0\right)\cdot\left(x_2-x_1,u_2-u_1\right)\\
		&\qquad \geq C_{f}|u_2-u_1|^2,\qquad (t,x^0)\in[0,T]\times\br,\quad x_1,x_2,u_1,u_2\in\br.
	\end{split}
\end{equation*}

\textbf{(H4)} For any $\gamma\in \pr_2(\br^2)$ with marginals $m,m'\in\pr_2(\br)$,
\begin{equation*}
	\begin{split}
		&\int_{\br^2}\left[\left(f^2_{x}\left(t,x,m,x^0\right)-f^2_{x}\left(t,y,m',x^0\right)\right)(x-y)\right]\gamma(dx,dy)\geq 0,\quad \left(t,x^0\right)\in[0,T]\times\br,\\
		&\int_{\br^2}\left[\left(g_x\left(x,m,x^0\right)-g_x\left(y,m',x^0\right)\right)(x-y)\right]\gamma(dx,dy)\geq 0,\quad x^0\in\br.   
	\end{split}
\end{equation*}
Equivalently, for any square-integrable random variables $\xi$ and $\xi'$ on the same probability space,
\begin{equation*}
	\begin{split}
		&\e \left[\left(f^2_{x}\left(t,\xi',\lr(\xi'),x^0\right)-f^2_{x}\left(t,\xi,\lr(\xi),x^0\right)\right)\left(\xi'-\xi\right)\right]\geq 0,\quad \left(t,x^0\right)\in [0,T]\times \br,\\
		&\e \left[\left(g_x\left(\xi',\lr(\xi'),x^0\right)-g_x\left(\xi,\lr(\xi),x^0\right)\right)\left(\xi'-\xi\right)\right]\geq 0,\quad x^0\in \br.    
	\end{split}
\end{equation*}

One of our main results is stated as follows.

\begin{theorem}\label{thm_main_1}
	Let Assumptions (H1)-(H4) be satisfied. Then, there exist $\delta>0$ depending only on $(L,T)$, such that
	\begin{align*}
		\left(\hat{u}^0,\hat{u}\right):=\left\{\left(\hat{u}^0\left(t,{X}^0_t,{p}^0_t,{q}^0_t,\lr(X_t|{\f}_t^0)\right),\hat{u}\left(t,{X}_t,{p}_t,{q}_t,{\tilde{q}}_t,X^0_t\right)\right),\ 0\le t\le T\right\}
	\end{align*}
	is the unique solution of Problem~\ref{def2} when $L_mC_f^{-1}\vee l_{x^0}l_m\le\delta$, where 
	\begin{align*}
		\left({X}^0,{p}^0,{q}^0,X,p,q,\tilde{q}\right)\in & \left(\sr_{\f^0}^2\times\sr_{\f^0}^2 \times\lr_{\f^0}^2\times\sr_{\f}^2\times\sr_{\f}^2\times\lr_{\f}^2\times\lr_{\f}^2\right)(0,T)
	\end{align*}
	is the unique solution of the following FBSDEs
	\begin{equation}\label{fbsde2}
		\left\{
		\begin{aligned}
			&dX^0_t=b^0\left(t,X^0_t,\hat{u}^0(t,X^0_t,p^0_t,q^0_t,\lr(X_t|{\f}_t^0)),\lr(X_t|{\f}_t^0)\right)dt\\
			&\qquad\quad+\sigma^0\left(t,X^0_t,\hat{u}^0(t,X^0_t,p^0_t,q^0_t,\lr(X_t|{\f}_t^0)),\lr(X_t|{\f}_t^0)\right)dW^0_t,\quad t\in(0,T];\\
			&dp^0_t=- H_{x^0}^0\left(t,X^0_t,p^0_t,q^0_t,\hat{u}^0(t,X^0_t,p^0_t,q^0_t,\lr(X_t|{\f}_t^0)),\lr(X_t|{\f}_t^0)\right)dt\\
			&\qquad\quad+q_t^0dW_t^0,\quad t\in[0,T);\\
			&dX_t=b\left(t,X_t,\hat{u}(t,X_t,p_t,q_t,\tilde{q}_t,X^0_t),\lr(X_t|{\f}_t^0)\right)dt\\
			&\qquad\quad+\sigma\left(t,X_t,\hat{u}(t,X_t,p_t,q_t,\tilde{q}_t,X_t^0),\lr(X_t|{\f}_t^0)\right)dW_t\\
			&\qquad\quad+\tilde{\sigma}\left(t,X_t,\hat{u}(t,X_t,p_t,q_t,\tilde{q}_t,X^0_t),\lr(X_t|{\f}_t^0)\right)d{W}^0_t,\quad t\in(0,T];\\
			&dp_t=- H_x\left(t,X_t,p_t,q_t,\tilde{q}_t,\hat{u}(t,X_t,p_t,q_t,\tilde{q}_t,X^0_t),\lr(X_t|{\f}_t^0),{X}^0_t\right)dt\\
			&\qquad\quad+q_tdW_t+\tilde{q}_td{W}^0_t,\quad t\in[0,T);\\
			&X^0_0=\xi^0,\  p^0_T=g_{x^0}^0(X^0_T,\lr(X_T|{\f}_T^0)),\  X_0=\xi,\  p_T=g_x(X_T,\lr(X_T|{\f}_T^0),X^0_T).
		\end{aligned}
		\right.
	\end{equation}
	Here, the Hamiltonians $(H^0,H)$ are defined in \eqref{H} and the functions $(\hat{u}^0(\cdot),\hat{u}(\cdot))$ are defined in \eqref{u}.
\end{theorem}

The solvability of FBSDEs \eqref{fbsde2} is proved in Theorem~\ref{main2_thm}. The existence result in Theorem~\ref{thm_main_1} is consequence of Theorem~\ref{thm:smp} and the existence of solutions of FBSDEs \eqref{fbsde2}. The uniqueness result in Theorem~\ref{thm_main_1} is a consequence of Theorem~\ref{thm:1} and the uniqueness of solutions of FBSDEs \eqref{fbsde2}. To solve Problem~\ref{def1}, we need the following additional assumption:

\textbf{(H5)} The functions $\left(b^0_0,\sigma_0^0,b_0,\sigma_0,\tilde{\sigma}_0,f^0,g^0,f,g\right)$ depend on $m\in\pr(\br)$ in a scalar form: for $\phi=b^0_0,\sigma_0^0,b_0,\sigma_0,\tilde{\sigma}_0,f^0,g^0,f,g$,
\begin{align*}
	&\phi(t,x,u,m,x^0)=\int_\br \overline{\phi}\left(t,x,u,y,x^0\right)m(dy),\\ 
	&\qquad \left(t,x,u,m,x^0\right)\in[0,T]\times\br\times\br\times\pr(\br)\times\br,
\end{align*}
with $\overline{\phi}$ being $L$-Lipschitz contimuous in $y\in\br$. And functions  $(\overline{f^0},\overline{g^0},\overline{f},\overline{g})$ satisfy the convexity conditions in $(H3)$.

Then, we have the following result for  Problem~\ref{def1}.
\begin{theorem}\label{thm_main_2}
	Let Assumptions (H1)-(H5) be satisfied. Then, a solution of Problem~\ref{def2} forms an $\mathcal{O}(\frac{1}{\sqrt{N}})$-Nash equilibrium for Problem~\ref{def1}.
\end{theorem}

Theorem~\ref{thm_main_2} is a consequence of Theorem~\ref{thm:nash}. In the next two sections, we solve Problems~\ref{def2} and \ref{def1}, respectively, and give the proofs of Theorems~\ref{thm_main_1} and \ref{thm_main_2}.

%%%%%%%%%%%%%%%%%%%%%%%%%%%%%%%%%%%%%%%%%%%%%%%%%%%%%%%%%%%%%%%%%%%%%%%%%%%
\section{Solvability of Problem~\ref{def2}}\label{pro1}

In this section, we consider Problem~\ref{def2} and give the proof of Theorem~\ref{thm_main_1}. We always suppose that assumptions (H1)-(H4) hold true. We begin with  the stochastic maximum principle for Problem~\ref{def2}. 

\subsection{Stochastic Maximum Principle}\label{SMP}
The stochastic maximum principle gives optimality conditions for the existence of an optimal control in terms of solvability of the adjoint process as a backward stochastic differential equation (BSDE). For more details about stochastic maximum principle and Pontryagin principle, we refer to \cite{HP,JYXY}. We first define the following two stochastic control problems $\mathbf{P}^m$ and $\mathbf{P}^{X^0,m}$. Given $m\in\pr_{\f^0}^2(0,T)$, $\mathbf{P}^m$ is defined as 
\begin{equation*}
	\left\{
	\begin{aligned}
		&\hat{u}^0\in\argmin_{u^0\in\lr^2_{\f^0}(0,T)}\e\left[\int_0^T f^0\left(t,X^{0,u^0}_t,u^0_t,m_t\right)dt+g^0\left(X^{0,u^0}_T,m_T\right)\right],\\
		&X^{0,u^0}_t=\xi^0+\int_0^t b^0\left(s,X^{0,u^0}_s,u^0_s,m_s\right)ds+\int_0^t\sigma^0\left(s,X^{0,u^0}_s,u^0_s,m_s\right)dW^0_s,\quad t\in[0,T].
	\end{aligned}
	\right.
\end{equation*}
Given $(m,X^0)\in\pr_{\f^0}^2(0,T)\times\sr_{\f^0}^2(0,T)$, $\mathbf{P}^{X^0,m}$ is defined as
\begin{equation*}
	\left\{
	\begin{aligned}
		&\hat{u}\in\argmin_{u\in\lr^2_{\f}(0,T)}\e\left[\int_0^T f\left(t,X^u_t,u_t,m_t,X_t^{0}\right)dt+g\left(X^u_T,m_T,X_T^{0}\right)\right],\\
		&X^u_t=\xi+\int_0^t b\left(s,X^u_s,u_s,m_s\right)ds+\int_0^t\sigma\left(s,X^u_s,u_s,m_s\right)dW_s\\
		&\qquad\quad\  +\int_0^t\tilde{\sigma}\left(s,X^u_s,u_s,m_s\right)dW^0_s,\quad t\in[0,T].
	\end{aligned}
	\right.
\end{equation*}
We define the generalized Hamiltonians 
\begin{equation}\label{H}
	\begin{split}
		&H^0(t,x^0,p^0,q^0,u^0,m)=b^0(t,x^0,u^0,m)p^0+\sigma^0(t,x^0,u^0,m)q^0+f^0(t,x^0,u^0,m), \\
		&H(t,x,p,q,\tilde{q},u,m,x^0)=b(t,x,u,m)p+\sigma(t,x,u,m)q+\tilde{\sigma}(t,x,u,m)\tilde{q}\\
		&\quad\qquad\qquad\qquad\qquad\qquad +f(t,x,u,m,x^0),
	\end{split}
\end{equation}
and the minimizing control functions 
\begin{equation}\label{u}
	\begin{split}
		&\hat{u}^0(t,x^0,p^0,q^0,m):=\argmin_{u^0\in \br}H^0(t,x^0,p^0,q^0,u^0,m),  \\
		&\qquad(t,x^0,p^0,q^0,m)\in [0,T]\times \mathbb{R}^3\times\pr_2(\br),\\
		&\hat{u}(t,x,p,q,\tilde{q},x^0):=\argmin_{u\in \br}H(t,x,p,q,\tilde{q},u,m,x^0),\\
		&\qquad(t,x^0,x,p,q,\tilde{q},m)\in [0,T]\times \mathbb{R}^5\times\pr_2.
	\end{split}
\end{equation}
Assumptions (H1) and (H3) ensure that the Hamiltonians $H^0$ and $H$ are strictly convex in $u^0$ and $u$, respectively, so that there is a unique minimizer. The separability condition of $f$ in Assumption (H2) ensures that $\hat{u}$ is independent of $m$. The following result is borrowed from \cite[Lemma 1]{PA}.

\begin{lemma}\label{mm:lemma:u}
	Let Assumptions (H1)-(H3) be satisfied. Then, the function $\hat{u}^0$ is measurable, locally bounded, $L(2C_{f^0})^{-1}$-Lipschitz continuous in $(x^0,p^0,q^0)$ and ${l_m}(2C_{f^0})^{-1}$-Lipschitz continuous in $m$; and the function $\hat{u}$ is measurable, locally bounded, $L(2C_f)^{-1}$-Lipschitz-continuous in $(x,p,q,\tilde{q})$ and $l_{x^0}(2C_f)^{-1}$-Lipschitz continuous in $x^0$.
\end{lemma}

We have the following solvability of $\mathbf{P}^m$ and $\mathbf{P}^{X^0,m}$.

\begin{lemma}\label{thm:mp}
	Let Assumptions (H1)-(H3) be satisfied. Given $m\in\pr_{\f^0}^2(0,T)$, then, $\hat{u}^0=\{\hat{u}^0(t,{X}^0_t,{p}^0_t,{q}^0_t,m_t),0\le t\le T\}$ is the unique optimal control of $\mathbf{P}^m$, where $({X}^0,{p}^0,{q}^0)\in (\sr^2_{\f^0}\times\sr^2_{\f^0}\times\lr^2_{\f^0})(0,T)$ is the solution of the following FBSDEs
	\begin{equation}\label{fbsde_equation''}
		\left\{
		\begin{aligned}
			&dX^0_t=b^0(t,X^0_t,\hat{u}^0(t,X^0_t,p^0_t,q^0_t,m_t),m_t)dt\\
			&\qquad\quad+\sigma^0(t,X^0_t,\hat{u}^0(t,X^0_t,p^0_t,q^0_t,m_t),m_t)dW^0_t,\quad t\in(0,T];\\
			&dp^0_t=- H_{x^0}^0(t,X^0_t,p^0_t,q^0_t,\hat{u}^0(t,X^0_t,p^0_t,q^0_t,m_t),m_t)dt+q_t^0dW_t^0,\quad t\in[0,T);\\
			&X^0_0=\xi^0, \quad p^0_T=g_{x^0}^0(X^0_T,m_T).
		\end{aligned}
		\right.   
	\end{equation}	
	Given $(m,X^0)\in\pr_{\f^0}^2(0,T)\times\sr_{\f^0}^2(0,T)$, then, $\hat{u}=\{\hat{u}(t,{X}_t,{p}_t,{q}_t,{\tilde{q}}_t,X^0_t),0\le t\le T\}$ is the unique optimal control of $\mathbf{P}^{X^0,m}$, where $({X},{p};{q},{\tilde{q}})\in (\sr^2_{\f})^2\times(\lr^2_{\f})^2(0,T)$ is the solution of the following FBSDEs
	\begin{equation}\label{fbsde_equation'}
		\left\{
		\begin{aligned}
			&dX_t=b(t,X_t,\hat{u}(t,X_t,p_t,q_t,\tilde{q}_t,X^0_t),m_t)dt\\
			&\qquad\quad+\sigma(t,X_t,\hat{u}(t,X_t,p_t,q_t,\tilde{q}_t,X_t^0),m_t)dW_t\\
			&\qquad\quad+\tilde{\sigma}(t,X_t,\hat{u}(t,X_t,p_t,q_t,\tilde{q}_t,X^0_t),m_t)d{W}^0_t,\quad t\in(0,T];\\
			&dp_t=- H_x(t,X_t,p_t,q_t,\tilde{q}_t,\hat{u}(t,X_t,p_t,q_t,\tilde{q}_t,X^0_t),m_t,{X}^0_t)dt\\
			&\qquad\quad+q_tdW_t+\tilde{q}_td{W}^0_t,\quad t\in[0,T);\\
			&X_0=\xi,\quad p_T=g_x(X_T,m_T,X^0_T).
		\end{aligned}
		\right.   
	\end{equation}
\end{lemma}

\begin{proof}
	The stochastic maximum principle for $\mathbf{P}^m$ and $\mathbf{P}^{X^0,m}$ is a direct consequence of \cite[Theorem 3.1]{HT}, and the existence and uniqueness of solutions of FBSDEs \eqref{fbsde_equation''} and \eqref{fbsde_equation'} is an immediate consequence of \cite[Theorem 2.3]{SP}.
\end{proof}

Now we turn to Problem~\ref{def2}. Given $m\in\pr_{\f^0}^2(0,T)$, we denote by $\hat{u}^{0}$ the optimal control of $\mathbf{P}^m$ and $\hat{X}^{0}$ the corresponding state process. We denote by $\hat{u}$ the optimal control of $\mathbf{P}^{\hat{X}^0,m}$ and $\hat{X}$ the corresponding state process. The definition of Problem~\ref{def2} states that, $(\hat{u}^0,\hat{u})$ is an optimal control of Problem~\ref{def2} if the following consistency condition is satisfied
\begin{equation*}
	m_t=\lr(\hat{X}_t|\f^0_t),\quad t\in[0,T].
\end{equation*}
Therefore, we have the following stochastic maximum principle for Problem~\ref{def2}. We first give the sufficient condition, which is a direct consequence of Lemma~\ref{thm:mp} and the definition of Problem~\ref{def2}.

\begin{theorem}\label{thm:smp}
	Let Assumptions (H1)-(H3) be satisfied. If FBSDEs \eqref{fbsde2} has a solution 
	\begin{align*}
		(X^0,p^0;q^0;X,p;q,\tilde{q})\in& (\sr^2_{\f^0})^2\times\lr^2_{\f^0}\times(\sr^2_{\f})^2\times(\lr^2_{\f})^2(0,T),
	\end{align*}
	then, $(\hat{u}^0,\hat{u}):=\{(\hat{u}^0(t,{X}^0_t,{p}^0_t,{q}^0_t,\lr(X_t|{\f}_t^0)),\hat{u}(t,{X}_t,{p}_t,{q}_t,{\tilde{q}}_t,X^0_t)),\ 0\le t\le T\}$ is a solution of Problem~\ref{def2}.
\end{theorem}

We next give a necessary condition. 

\begin{theorem}\label{thm:1}
	Let Assumptions (H1)-(H3) be satisfied. Suppose that $(\hat{u}^0,\hat{u})\in(\lr_{\f^0}^2\times\lr_{\f}^2)(0,T)$ is a solution of Problem~\ref{def2}, $(\hat{X}^0,\hat{X})\in (\sr_{\f^0}^2\times\sr_{\f}^2)(0,T)$ is the corresponding state processes, $(\hat{p}^0,\hat{q}^0)\in(\sr_{\f^0}^2\times\lr_{\f^0}^2)(0,T)$ and $(\hat{p};\hat{q},\hat{\tilde{q}})\in\sr_{\f}^2\times(\lr_{\f}^2)^2(0,T)$ are adjoint processes defined as
	\begin{equation}\label{thm1_1}
		\begin{split}
			&\left\{
			\begin{aligned}
				&d\hat{p}^0_t=- H_{x^0}^0(t,\hat{X}^0_t,\hat{p}^0_t,\hat{q}^0_t,\hat{u}^0_t,\lr(\hat{X}_t|{\f}_t^0))dt+\hat{q}_t^0dW_t^0,\quad t\in[0,T),\\
				&\hat{p}^0_T=g_{x^0}^0(\hat{X}^0_T,\lr(\hat{X}_T|{\f}_T^0)),
			\end{aligned}
			\right.\\ \\
			&\left\{
			\begin{aligned}
				&d\hat{p}_t=- H_x(t,\hat{X}_t,\hat{p}_t,\hat{q}_t,\hat{\tilde{q}}_t,\hat{u}_t,\lr(\hat{X}_t|{\f}_t^0),\hat{X}^0_t)dt+\hat{q}_tdW_t+\hat{\tilde{q}}_td{W}^0_t,\  t\in[0,T),\\
				&\hat{p}_T=g_x(\hat{X}_T,\lr(\hat{X}_T|{\f}_T^0),\hat{X}^0_T).
			\end{aligned}
			\right.
		\end{split}
	\end{equation}
	Then, we have 
	\begin{equation}\label{thm1_2}
		\begin{split}
			&\hat{u}^0_t=\hat{u}^0(t,\hat{X}^0_t,\hat{p}^0_t,\hat{q}^0_t,\lr(\hat{X}_t|{\f}_t^0)),\\
			&\hat{u}_t=\hat{u}(t,\hat{X}_t,\hat{p}_t,\hat{q}_t,\hat{\tilde{q}}_t,\hat{X}^0_t).
		\end{split}		
	\end{equation}
	That is, $(\hat{X}^0,\hat{p}^0,\hat{q}^0,\hat{X},\hat{p},\hat{q},\hat{\tilde{q}})$ is a solution of FBSDEs \eqref{fbsde2}.
\end{theorem}

\begin{proof}
	The solvability of backward equations \eqref{thm1_1} can be found in \cite{JYXY}. We denote by $\hat{m}_t:=\lr(\hat{X}_t|{\f}_t^0)$ for $t\in[0,T]$. Since that $(\hat{u}^0,\hat{u})$ is a solution of Problem~\ref{def2}, we know that $\hat{u}^0$ is an optimal control of $\mathbf{P}^{\hat{m}}$ and $\hat{u}$ is an optimal control of $\mathbf{P}^{\hat{X}^0,\hat{m}}$. We denote by $(X^0,p^0,q^0)$ the solution of FBSDEs \eqref{fbsde_equation''} corresponding to $\hat{m}$, and denote by $(X,p,q,\tilde{q})$ the solution of FBSDEs \eqref{fbsde_equation'} corresponding to $(\hat{m},\hat{X}^0)$. From the uniqueness in Lemma~\ref{thm:mp}, we know that	
	\begin{equation}\label{thm1_3}
		\begin{split}
			&\hat{u}^0_t=\hat{u}^0(t,{X}^0_t,{p}^0_t,{q}^0_t,\hat{m}_t),\quad \hat{u}_t=\hat{u}(t,{X}_t,{p}_t,{q}_t,{\tilde{q}}_t,\hat{X}^0_t).
		\end{split}
	\end{equation}
	Then, we know that $(X^0,X)$ coincide with $(\hat{X}^0,\hat{X})$ in $(\sr_{\f^0}^2\times\sr_{\f}^2)(0,T)$, and then, $(p^0;q^0;p;q,\tilde{q})$ coincide with $(\hat{p}^0;\hat{q}^0;\hat{p};\hat{q},\hat{\tilde{q}})$ in $\sr_{\f^0}^2\times\lr_{\f^0}^2\times \sr_{\f}^2\times(\lr_{\f}^2)^2(0,T)$. Therefore, we obtain \eqref{thm1_2} from \eqref{thm1_3}.
\end{proof}

\subsection{Solvability of FBSDEs \eqref{fbsde2}}\label{M1}
In this subsection, we prove the existence and uniqueness of the solutions of FBSDEs \eqref{fbsde2} by the method of continuation in coefficients. For notational convenience, we denote by $\theta^0$ a process $(X^0,u^0)\in (\sr_{\f^0}^2\times\lr_{\f^0}^2)(0,T)$, and $\theta$ a process 
\begin{align*}
	\{(X_t,u_t,\lr(X_t|{\f}^0_t)),0\le t\le T\}\in (\sr_{\f}^2\times\lr_{\f}^2\times\pr_{\f^0}^2)(0,T).
\end{align*}
We denote by $\Theta^0$ a process $(\theta^0,p^0,q^0)$ with $(p^0,q^0)\in (\sr_{\f^0}^2\times\lr_{\f^0}^2)(0,T)$, and $\Theta$ a process $(\theta,p,q,\tilde{q})$ with $(p;q,\tilde{q})\in \sr_{\f}^2\times(\lr_{\f}^2)^2(0,T)$. We then denote by $\mathbb{S}$ the space of processes $(\Theta^0,\Theta)$ and
\begin{equation*}
	\begin{split}
		\|(\Theta^0,\Theta)\|_{\mathbb{S}}^2:&=\e\left[\sup_{0\le t\le T}|(X^0_t,p^0_t,X_t,p_t)|^2+\int_0^T|(u^0_t,q^0_t,u_t,q_t,\tilde{q}_t)|^2dt\right]<+\infty.	
	\end{split}
\end{equation*}
The  strategy to prove the solvability of FBSDEs \eqref{fbsde2} is to prove that the solvability are preserved when the coefficients are slightly perturbed. We denote by $(\ii^0,\ii)$ an input for FBSDEs \eqref{fbsde2} with
\begin{align*}
	&\ii^0=(\ii^{b^0},\ii^{\sigma^0},\ii^{f^0},\ii_T^{g^0})\in (\lr_{\f^0}^2\times\lr_{\f^0}^2\times\lr_{\f^0}^2)(0,T)\times\lr_{\f^0_T}^2,\\
	&\ii=(\ii^b,\ii^{\sigma},\ii^{\tilde{\sigma}},\ii^f,\ii_T^g)\in(\lr_{\f}^2\times\lr_{\f}^2\times\lr_{\f}^2\times\lr_{\f}^2)(0,T)\times\lr_{\f_T}^2,
\end{align*}
and $\mathbb{I}$ the space of all inputs, endowed with the squared norm
\begin{equation*}
	\|(\ii^0,\ii)\|_{\iii}^2:=\e\left[|(\ii_T^{g^0},\ii_T^g)|^2+\int_0^T |(\ii_t^{b^0},\ii_t^{\sigma^0},\ii_t^{f^0},\ii_t^b,\ii_t^{\sigma},\ii_t^{\tilde{\sigma}},\ii_t^f)|^2 dt\right].
\end{equation*}
For any $(\gamma,\xi^0,\xi,\ii^0,\ii)\in[0,1]\times \lr_{\f^0_0}^2\times \lr_{\f_0}^2\times\iii$, we denote by $\mathcal{E}(\gamma,\xi^0,\xi,\ii^0,\ii)$ the FBSDEs
\begin{equation}\label{main2'}
	\left\{
	\begin{aligned}
		&dX^0_t=(\gamma b^0(t,\theta^0_t,m_t)+\ii_t^{b^0})dt+(\gamma\sigma^0(t,\theta^0_t,m_t)+\ii_t^{\sigma^0})dW^0_t,\quad t\in(0,T];\\
		&dp^0_t=-(\gamma H_{x^0}^0(t,\Theta^0_t,m_t)+\ii_t^{f^0})dt+q^0_tdW^0_t,\quad t\in[0,T);\\
		&dX_t=(\gamma b(t,\theta_t)+\ii_t^b)dt+(\gamma\sigma(t,\theta_t)+\ii_t^{\sigma})dW_t\\
		&\qquad\quad+(\gamma\tilde{\sigma}(t,\theta_t)+\ii_t^{\tilde{\sigma}})dW^0_t,\quad t\in(0,T];\\
		&dp_t=-(\gamma H_x(t,\Theta_t,X_t^0)+\ii_t^f)dt+q_tdW_t+\tilde{q}_td\tilde{W}_t,\quad t\in[0,T),\\
		&X_0^0=\xi^0,\  X_0=\xi,\  p^0_T=\gamma g_{x^0}^0(X^0_T,m_T)+\ii_T^{g_0},\  p_T=\gamma g_x(X_T,m_T,X^0_T)+\ii_T^g,\\
		&u^0_t=\hat{u}^0(t,X^0_t,p^0_t,q^0_t,m_t),\quad u_t=\hat{u}(t,X_t,p_t,q_t,\tilde{q}_t,X^0_t),\quad t\in[0,T],\\
		&m_t=\lr(X_t|{\f}^0_t),\quad  t\in[0,T].\\
	\end{aligned}
	\right.
\end{equation}
Note that $\mathcal{E}(1,\xi^0,\xi,0,0)$ is just  FBSDEs \eqref{fbsde2}. For $\gamma\in [0,1]$, we say that property $(\mathcal{S}_\gamma)$ holds true if, for any $(\xi^0,\xi,\ii^0,\ii)\in\lr_{\f^0_0}^2\times \lr_{\f_0}^2\times\iii$, equation $\mathcal{E}(\gamma,\xi^0,\xi,\ii^0,\ii)$ has a unique solution in $\s$. Our aim is to show $(\mathcal{S}_1)$ holds true. We first have the following lemma,  which is crucial in the proof of our main result Theorem~\ref{main2_thm}.

\begin{lemma}\label{main2_lem2}
	Let Assumptions (H1)-(H4) be satisfied and $\gamma\in [0,1)$ such that $(\mathcal{S}_\gamma)$ holds true. Then, there exist positive constants $(\delta,\eta_0)$ depending only on $(L,T)$, such that $(\mathcal{S}_{\gamma+\eta})$ holds true for any $\eta\in(0,\eta_0\wedge(1-\gamma)]$ when $L_mC_f^{-1}\vee l_{x^0}l_m\le\delta$.
\end{lemma}

\begin{proof}
	Let $\gamma\in[0,1)$ such that $(\mathcal{S}_{\gamma})$ holds true. For $(\xi^0,\xi,\ii^0,\ii)\in \lr_{\f^0_0}^2\times\lr_{\f_0}^2\times\iii$ and $\eta>0$, we now show that the FBSDEs $\mathcal{E}(\gamma+\eta,\xi_0,\xi,\ii^0,\ii)$ has a unique solution in $\mathbb{S}$. We define a map $\Phi:\mathbb{S}\to\mathbb{S}$ as follows: for any $(\Theta^0,\Theta)\in\mathbb{S}$, let $\Phi(\Theta^0,\Theta)$  be the unique solution of  equation $\mathcal{E}(\gamma,\xi_0,\xi,\hat{\ii}^0,\hat{\ii})$, where
	\begin{align*}
		&(\hat{\ii}_t^{b^0},\hat{\ii}_t^{\sigma^0}):=\eta (b,\sigma)(t,\theta_t^0,m_t)+(\ii_t^{b^0},\ii_t^{\sigma^0}),\\
		&(\hat{\ii}_t^{b},\hat{\ii}_t^{\sigma},\hat{\ii}_t^{\tilde{\sigma}}):=\eta(b,\sigma,\tilde{\sigma})(t,\theta_t)+(\ii_t^{b},\ii_t^{\sigma},\ii_t^{\tilde{\sigma}}),\\
		&\hat{\ii}_t^{f^0}:=\eta H_{x^0}^0(t,\Theta^0_t,m_t)+\ii_t^{f^0},\quad \hat{\ii}_t^{f}:=\eta H_x(t,\Theta_t,X_t^0)+\ii_t^{f},\\
		&\hat{\ii}_T^{g^0}:=\eta g^0_{x^0}(X^0_T,m_T)+\ii_T^{g^0},\quad \hat{\ii}_T^{g}:=\eta g_x(X_T,m_T,X_t^0)+\ii_T^{g}.
	\end{align*}
	It is obvious that $(\Theta^0,\Theta)\in\mathbb{S}$ is a fixed point of $\Phi$ if and only if $(\Theta^0,\Theta)$ is a solution of $\mathcal{E}(\gamma+\eta,\xi_0,\xi,\ii^{0},\ii)$. Therefore, we only need to show that $\Phi$ is a contraction when $\eta$ is small enough. In fact, we know from Lemma~\ref{main2_lem1} below and Assumptions (H1)-(H2) that, there exist positive constants $(\delta,C)$ depending only on $(L,T)$, such that for any $(\Theta^{0}_1,\Theta_1),(\Theta^{0}_2,\Theta_2)\in\mathbb{S}$, 
	\begin{equation*}
		\|\Phi(\Theta^{0}_2,\Theta_2)-\Phi(\Theta^{0}_1,\Theta_1)\|_{\mathbb{S}}\le C\|(\hat{\ii}^{0}_2-\hat{\ii}^{0}_1,\hat{\ii}_{2}-\hat{\ii}_{1})\|_{\iii}\le C\eta\|(\Theta^{0}_2-\Theta^{0}_1,\Theta_{2}-\Theta_{1})\|_{\mathbb{S}},
	\end{equation*}
	when $L_mC_f^{-1}\vee l_{x^0}l_m\le\delta$. So when $\eta$ is small enough, $\Phi$ is a contraction.
\end{proof}

Our main result of this section is stated as follows.

\begin{theorem}\label{main2_thm}
	Let Assumptions (H1)-(H4) be satisfied. Then, there is a constant $\delta>0$ depending only on $(L,T)$, such that FBSDEs \eqref{fbsde2} is uniquely solvable when $L_mC_f^{-1}\vee l_{x^0}l_m\le\delta$.
\end{theorem}

\begin{proof}
	It is a direct consequence of Lemma~\ref{main2_lem2} and the fact that $(\mathcal{S}_0)$ holds true.
\end{proof}

Theorem~\ref{thm_main_1} is now a direct consequence of Theorems~\ref{thm:smp} and \ref{main2_thm}. To complete the previous proof, it remains to state and prove the following lemma.

\begin{lemma}\label{main2_lem1}
	Let Assumptions (H1)-(H4) be satisfied and $\gamma\in [0,1]$ such that $(\mathcal{S}_\gamma)$ holds true. Then, there exist positive constants $(\delta,C)$ depending only on $(L,T)$, such that for any $(\xi_1^0,\xi_1,\ii^{0}_1,\ii_1),(\xi^0_2,\xi_2,\ii^0_2,\ii_2)\in \lr_{\f^0_0}^2\times\lr_{\f_0}^2\times\iii$, the solutions $(\Theta^{0}_1,\Theta_1)$ and $(\Theta^{0}_2,\Theta_2)$ of $\mathcal{E}(\gamma,\xi^{0}_1,\xi_1,\ii^{0}_1,\ii_1)$ and $\mathcal{E}(\gamma,\xi^{0}_2,\xi_2,\ii^{0}_2,\ii_2)$ satisfy
	\begin{equation*}
		\|(\Theta^{0}_1-\Theta^{0}_2,\Theta_{1}-\Theta_{2})\|_{\mathbb{S}}^2\le C\big(\e[|\xi_{1}^0-\xi^{0}_2|^2+|\xi_1-\xi_2|^2]+\|(\ii^{0}_1-\ii^{0}_2,\ii_1-\ii_2)\|_{\iii}^2\big)
	\end{equation*}
	when $L_mC_f^{-1}\vee l_{x^0}l_m\le\delta$.
\end{lemma}

\begin{proof}
	We set $(\Delta \Theta^0,\Delta \Theta)=(\Theta^{0}_2-\Theta^{0}_1,\Theta_{2}-\Theta_{1})$ and $ (\Delta\ii^0,\Delta\ii)=(\ii^0_2-\ii^0_1,\ii_2-\ii_1)$. Then $(\Delta \Theta^0,\Delta \Theta)$ satisfy the following FBSDEs
	\begin{equation*}
		\left\{
		\begin{aligned}
			&d\Delta X^0_t=[\gamma((b^0_0(t,m_t^2)-b^0_0(t,m_t^1))+b^0_1(t)\Delta X^0_t+b^0_2(t)\Delta u^0_t)+\Delta\ii_t^{b^0}]dt\\
			&\qquad\qquad+[\gamma((\sigma^0_0(t,m_t^2)-\sigma^0_0(t,m_t^1))+\sigma^0_1(t)\Delta X^0_t+\sigma^0_2(t)\Delta u^0_t)+\Delta\ii_t^{\sigma^0}]dW^0_t,\  t\in(0,T];\\
			&d\Delta p^0_t=-[\gamma(b^0_1(t)\Delta p^0_t+\sigma^0_1(t)\Delta q^0_t+f^0_{x^0}(t,X_t^{0,2},u_t^{0,2},m_t^2)\\
			&\qquad\qquad-f^0_{x^0}(t,X_t^{0,1},u_t^{0,1},m_t^1))+\Delta\ii_t^{f^0}]dt+\Delta q^0_tdW^0_t,\quad  t\in[0,T);\\
			&d\Delta X_t=[\gamma((b_0(t,m_t^2)-b_0(t,m_t^1))+b_1(t)\Delta X_t+b_2(t)\Delta u_t)+\Delta\ii_t^b]dt\\
			&\qquad\qquad+[\gamma((\sigma_0(t,m_t^2)-\sigma_0(t,m_t^1))+\sigma_1(t)\Delta X_t+\sigma_2(t)\Delta u_t)+\Delta\ii_t^{\sigma}]dW_t\\
			&\qquad\qquad+[\gamma((\tilde{\sigma}_0(t,m_t^2)-\tilde{\sigma}_0(t,m_t^1))+\tilde{\sigma}_1(t)\Delta X_t+\tilde{\sigma}_2(t)\Delta u_t)+\Delta\ii_t^{\tilde{\sigma}}]d{W}^0_t,\  t\in(0,T];\\
			&d\Delta p_t=-[\gamma(b_1(t)\Delta p_t+\sigma_1(t)\Delta q_t+\tilde{\sigma}_1(t)\Delta\tilde{q}_t+f_x(t,X_t^2,u_t^2,m_t^2,X_t^{0,2})\\
			&\qquad\qquad-f_x(t,X_t^1,u_t^1,m_t^1,X_t^{0,1}))+\Delta\ii_t^f]dt+\Delta q_tdW_t+\Delta\tilde{q}_td{W}^0_t,\quad t\in[0,T);\\
			&\Delta X^0_0=\Delta\xi^0,\quad \Delta p^0_T=\gamma(g^0_{x^0}(X^{0,2}_T,m_T^2)-g^0_{x^0}(X_T^{0,1},m_T^1))+\Delta\ii_T^{g^0},\\
			&\Delta X_0=\Delta\xi,\quad \Delta p_T=\gamma(g_x(X_T^2,m_T^2,X_T^{0,2})-g_x(X_T^1,m_T^1,X_T^{0,1}))+\Delta\ii_T^g
		\end{aligned}
		\right.
	\end{equation*}
	with $m_t^i=\lr(X_t^i|\f^0_t)$ for $i=1,2$, and the condition
	\begin{equation}\label{main1_lem1_0.2}
		\begin{split}
			& b_2^0(t)\Delta p_t^0+\Delta\sigma^0_2(t)q_t^0+f^0_{u}(t,X_t^{0,2},u_t^{0,2},m_t^2)-f^0_{u}(t,X_t^{0,1},u_t^{0,1},m_t^1)=0,\\
			& b_2(t)\Delta p_t+\sigma_2(t)\Delta q_t+\Delta\tilde{\sigma}_2(t)\tilde{q}_t+f^1_{u}(t,X^2_t,u^2_t,X^{0,2}_t)-f^1_{u}(t,X^1_t,u^1_t,X^{0,1}_t)=0.
		\end{split}
	\end{equation}
	From Assumptions (H1)-(H2) and standard estimates for SDEs and BSDEs, there exists $C_1>0$ depending only on $(L,T)$, such that
	\begin{equation}\label{main2_lem1_1}
		\begin{split}
			&\e\left[\sup_{0\le t\le T}|(\Delta X^0_t,\Delta p^0_t)|^2+\int_0^T |\Delta q^0_t|^2 dt\right]\\
			\le\ & C_1\left(\e\left[|\Delta\xi^0|^2+\gamma l_m^2\sup_{0\le t\le T}|\Delta X_t|^2+\gamma\int_0^T  |\Delta u^0_t|^2 dt\right]+\|\Delta\ii^0\|^2\right).
		\end{split}
	\end{equation}
	Here, we have used the inequality
	\begin{equation*}\label{main1_lem1_0.1}
		\e\left[W_2(m_t^1,m_t^2)^2\right]\le \e\left[\e[|\Delta X_t|^2|\f^0_t]\right]=\e[|\Delta X_t|^2], \quad t\in[0,T].
	\end{equation*}
	Applying It\^o's lemma on $\Delta p^0_t\Delta X^0_t$ and taking expectation, we have
	\begin{equation}\label{main2_lem1_4}
		\begin{split}
			&\e[\Delta p^0_T\Delta X^0_T]-\e[\Delta p^0_0\Delta\xi^0]\\
			=\ &\e\bigg[\int_0^T \gamma\Big[(b^0_0(t,m_t^2)-b^0_0(t,m_t^1))\Delta p^0_t+(\sigma^0_0(t,m_t^2)-\sigma^0_0(t,m_t^1))\Delta q^0_t\\
			&\qquad\qquad +(b^0_2(t)\Delta p^0_t+\sigma^0_2(t)\Delta q^0_t)\Delta u^0_t\\
			&\qquad\qquad -(f^0_{x^0}(t,X_t^{0,2},u_t^{0,2},m_t^2)-f^0_{x^0}(t,X_t^{0,1},u_t^{0,1},m_t^1))\Delta X^0_t\Big]\\
			&\qquad\quad +\left(\Delta p^0_t\Delta\ii_t^{b^0}+\Delta q^0_t\Delta\ii_t^{\sigma^0}-\Delta X^0_t\Delta\ii_t^{f^0}\right) dt\bigg].
		\end{split}
	\end{equation}
	From Assumption (H2), we have
	\begin{equation}\label{main1_lem1_8}
		\begin{split}
			&\e\left[\int_0^T\gamma\left[(b^0_0(t,m_t^2)-b^0_0(t,m_t^1))\Delta p^0_t+(\sigma^0_0(t,m_t^2)-\sigma^0_0(t,m_t^1))\Delta q^0_t\right]dt\right]\\
			\le\ & \e\big[\int_0^T\gamma l_m W_2(m_t^1,m_t^2)(|\Delta p_t^0|+|\Delta q_t^0|)dt\big].
		\end{split}
	\end{equation}
	From \eqref{main1_lem1_0.2} and Assumptions (H2)-(H3), we have
	\begin{equation}\label{main2_lem1_5}
		\begin{split}
			&(b^0_2(t)\Delta p^0_t+\sigma^0_2(t)\Delta q^0_t)\Delta u^0_t-(f^0_{x^0}(t,X_t^{0,2},u_t^{0,2},m_t^2)-f^0_{x^0}(t,X_t^{0,1},u_t^{0,1},m_t^1))\Delta X^0_t\\
			&=-\left[(f^0_{u^0},f^0_{x^0})(t,X_t^{02},u_t^{02},m_t^2)-(f^0_{u^0},f^0_{x^0})(t,X_t^{01},u_t^{01},m_t^2)\right]\cdot(\Delta u^0_t,\Delta X^0_t)\\
			&\quad\  +\left[(f^0_{u^0},f^0_{x^0})(t,X_t^{01},u_t^{01},m_t^2)-(f^0_{u^0},f^0_{x^0})(t,X_t^{01},u_t^{01},m_t^1)\right]\cdot(\Delta u^0_t,\Delta X^0_t)\\
			&\le -2C_{f^0}|\Delta u^0_t|^2+l_m W_2(m_t^1,m_t^2)(|\Delta u_t^0|+|\Delta X_t^0|).
		\end{split}
	\end{equation}
	From Assumptions (H2)-(H3), we have
	\begin{equation}\label{main2_lem1_6}
		\begin{split}
			\e\left[\Delta p_T^0\Delta X_T^0\right]=\ &\gamma\e\left[(g^0_{x^0}(X_T^{0,2},m_T^2)-g^0_{x^0}(X_T^{0,1},m_T^2))\Delta X_T^0\right]\\
			&+\e\left[\gamma(g^0_{x^0}(X_T^{0,1},m_T^2)-g^0_{x^0}(X_T^{0,1},m_T^1))\Delta X_T^0+\Delta\ii_T^{g^0}\Delta X_T^0\right]\\
			\geq\ & -\e\left[\gamma l_mW_2(m_T^1,m_T^2)|\Delta X_T^0|+|\Delta\ii_T^{g^0}||\Delta X_T^0|\right].
		\end{split}
	\end{equation}
	Plugging \eqref{main1_lem1_8}-\eqref{main2_lem1_6} into \eqref{main2_lem1_4}, we have for any $\epsilon\in(0,1)$,
	\begin{align*}
		&2\gamma C_{f^0}\e\big[\int_0^T |\Delta u_t^0|^2 dt\big]\\
		\le\ & \e\bigg[\gamma l_mW_2(m_T^1,m_T^2)|\Delta X_T^0|+|\Delta\ii_T^{g^0}||\Delta X_T^0|+|\Delta p_0^0||\Delta\xi_0|\\
		&\quad +\gamma\int_0^T l_m W_2(m_t^1,m_t^2)\left(|\Delta X_t^0|+|\Delta p_t^0|+|\Delta q_t^0|+|\Delta u_t^0|\right) dt\\
		&\quad +\int_0^T \left(\Delta p^0_t\Delta\ii_t^{b^0}+\Delta q^0_t\Delta\ii_t^{\sigma^0}-\Delta X^0_t\Delta\ii_t^{f^0}\right) dt\bigg]\\
		\le\ & \epsilon\e\left[(T+1)\sup_{0\le t\le T}|(\Delta X_t^0,\Delta p_t^0)|^2+\int_0^T \left(|\Delta q^0_t|^2+\gamma |\Delta u_t^0|^2 \right) dt\right]\\
		& +\frac{1}{\epsilon}\e\left[\gamma l_m^2(T+1)\sup_{0\le t\le T}|\Delta X_t|^2+|\Delta\xi^0|^2\right]+\frac{1}{\epsilon}\|\Delta\ii^0\|^2.
	\end{align*}
	Plugging \eqref{main2_lem1_1} into above, we have
	\begin{align*}
		&2\gamma C_{f^0}\e\left[\int_0^T |\Delta u_t^0|^2 dt\right]\\
		\le\ & \epsilon\gamma(T+1)(C_1+1)\e\left[\int_0^T |\Delta u^0_t|^2dt\right]+\gamma l_m^2(T+1)(\frac{1}{\epsilon}+C_1)\e\left[\sup_{0\le t\le T}|\Delta X_t|^2\right]\\
		& +\left(\frac{1}{\epsilon}+(T+1)C_1\right)\left(\|\Delta\ii^0\|^2+\e[|\Delta\xi^0|^2]\right).
	\end{align*}
	Choosing $\epsilon=C_{f^0}(T+1)^{-1}(C_1+1)^{-1}$,  we have
	\begin{equation}\label{main2_lem1_8}
		\begin{split}
			&\gamma \e\left[\int_0^T |\Delta u_t^0|^2 dt\right]\\
			\le\ & \gamma l_m^2L^2(T+1)^2(2C_1+1)\e\left[\sup_{0\le t\le T}|\Delta X_t|^2\right]\\
			& +L^2(T+1)(2C_1+1)\left(\|\Delta\ii^0\|^2+\e[|\Delta\xi^0|^2]\right).
		\end{split}
	\end{equation}
	Plugging \eqref{main2_lem1_8} into \eqref{main2_lem1_1}, we have
	\begin{align*}
		&\e\left[\sup_{0\le t\le T}|(\Delta X^0_t,\Delta p^0_t)|^2+\int_0^T |\Delta q^0_t|^2 dt\right]\\
		\le \ & 2\gamma l_m^2L^2(T+1)^2C_1(C_1+1)\e\left[\sup_{0\le t\le T}|\Delta X_t|^2\right] +C(L,T)\left(\|\Delta\ii^0\|^2+\e[|\Delta\xi^0|^2]\right),
	\end{align*}
	where $C(L,T)$ is a constant depending only on $(L,T)$. We know from Lemma~\ref{mm:lemma:u} that
	\begin{equation*}
		|\Delta u^0_t|\le \frac{L^2}{2}\left(|\Delta X^0_t|+|\Delta p^0_t|+|\Delta q^0_t|\right)+\frac{l_mL}{2} W_2(m_t^1,m_t^2),
	\end{equation*}
	so we eventually have
	\begin{equation}\label{main2_lem1_a}
		\begin{split}
			\|\Delta\Theta^0\|^2\le &6 l_m^2L^6(T+1)^3(C_1+1)^2\e\left[\sup_{0\le t\le T}|\Delta X_t|^2\right]\\
			&  +C(L,T)\left(\|\Delta\ii^0\|^2+\e[|\Delta\xi^0|^2]\right).
		\end{split}
	\end{equation}
	
	Now we give the estimates of $\Delta\Theta$. From Assumptions (H1)-(H2) and standard estimates for SDEs and BSDEs, there exists $C_2>0$ depending only on $(L,T)$, such that
	\begin{equation}\label{main2_lem1_10.5}
		\begin{split}
			&\e\left[\sup_{0\le t\le T}|(\Delta X_t,\Delta p_t)|^2+\int_0^T |(\Delta q_t,\Delta \tilde{q}_t)|^2 dt\right]\\
			\le\ & C_2\left(\e\left[|\Delta\xi|^2+\gamma l_{x^0}^2 \sup_{0\le t\le T}|\Delta X^0_t|^2+\gamma \int_0^T|\Delta u_t|^2 \right]+\|\Delta\ii\|^2\right).
		\end{split}
	\end{equation}
	Applying It\^o's lemma on $\Delta p_t\Delta X_t$ and taking expectation, we have
	\begin{equation}\label{main2_lem1_11}
		\begin{split}
			&\e[\Delta p_T\Delta X_T]-\e[\Delta p_0\Delta\xi]\\
			=\ &\e\bigg[\int_0^T \gamma \Big[(b_0(t,m_t^2)-b_0(t,m_t^1))\Delta p_t+(\sigma_0(t,m_t^2)-\sigma_0(t,m_t^1))\Delta q_t\\
			&\qquad\quad+(\tilde{\sigma}_0(t,m_t^2)-\tilde{\sigma}_0(t,m_t^1))\Delta \tilde{q}_t+(b_2(t)\Delta p_t+\sigma_2(t)\Delta q_t+\tilde{\sigma}_2(t)\Delta\tilde{q}_t)\Delta u_t\\
			&\qquad\quad -(f_x(t,X_t^2,u_t^2,m_t^2,X_t^{02})-f_x(t,X_t^1,u_t^1,m_t^1,X_t^{01}))\Delta X_t\Big]\\
			&\qquad\quad+(\Delta p_t\Delta\ii_t^b+\Delta q_t\Delta\ii_t^{\sigma}+\Delta\tilde{q}_t\Delta\ii_t^{\tilde{\sigma}}-\Delta X_t\Delta\ii_t^f) dt\bigg].
		\end{split}
	\end{equation}
	From Assumption (H2), we have
	\begin{equation}\label{main1_lem1_8'}
		\begin{split}
			&\e\bigg[\int_0^T\gamma\Big[(b_0(t,m_t^2)-b_0(t,m_t^1))\Delta p_t+(\sigma_0(t,m_t^2)-\sigma_0(t,m_t^1))\Delta q_t\\
			&\qquad\qquad+(\tilde{\sigma}_0(t,m_t^2)-\tilde{\sigma}_0(t,m_t^1))\Delta \tilde{q}_t\Big]dt\bigg]\\
			\le\ & \e\left[\int_0^T\gamma L_m W_2(m_t^1,m_t^2)\left(|\Delta p_t|+|\Delta q_t|+|\Delta \tilde{q}_t|\right)dt\right]\\
			\le\ & \gamma L_m(T+1)\e\left[\sup_{0\le t\le T}|(\Delta X_t,\Delta p_t)|^2+\int_0^T|(\Delta q_t,\Delta\tilde{q}_t)|^2dt\right].
		\end{split}
	\end{equation}
	From \eqref{main1_lem1_0.2} and Assumptions (H2)-(H3), we have
	\begin{equation}\label{main2_lem1_12}
		\begin{split}
			&(b_2(t)\Delta p_t+\sigma_2(t)\Delta q_t+\tilde{\sigma}(t)\Delta \tilde{q}_t)\Delta u_t\\
			&-(f^1_{x}(t,X_t^2,u_t^2,X_t^{0,2})-f^1_{x}(t,X_t^1,u_t^1,X_t^{0,1}))\Delta X_t \\
			\le\ & -2C_f|\Delta u_t|^2+l_{x^0}|\Delta X_t^0|(|\Delta u_t|+|\Delta X_t|).
		\end{split}
	\end{equation}
	From Assumptions (H2) and (H4), we have
	\begin{equation}\label{main1_lem1_5}
		\begin{split}
			&\e\left[(f_{x}^2(t,X_t^2,m_t^2,X_t^{0,2})-f_{x}^2(t,X_t^1,m_t^1,X_t^{0,1}))\Delta X_t\right]\\
			=\ &\e\left[\e[(f^2_{x}(t,X_t^2,m_t^2,X_t^{0,2})-f^2_{x}(t,X_t^1,m_t^1,X_t^{0,2}))\Delta X_t|{\f}^0_t]\right]\\
			&+\e\left[(f^2_{x}(t,X_t^1,m_t^1,X_t^{0,2})-f^2_{x}(t,X_t^1,m_t^1,X_t^{0,1}))\Delta X_t\right]\\\
			\geq\ & -l_{x^0}\e\left[|\Delta X_t^0||\Delta X_t|\right]\\
			&\e[\Delta p_T\Delta X_T]\\
			=\ & \e\left[\e\left[\gamma\left(g_x(X_T^2,m_T^2,X_T^{0,2})-g_x(X_T^1,m_T^1,X_t^{0,2})\right)\Delta X_T\Big|\f^0_T\right]\right]\\
			&+\e\left[\gamma(g_x(X_T^1,m_T^1,X_T^{0,2})-g_x(X_T^1,m_T^1,X_t^{0,1}))\Delta X_T+\Delta X_T\Delta\ii_T^g\right]\\		
			\geq\ & -\gamma l_{x^0} \e \left[|\Delta X_T^0||\Delta X_T|\right]+\e\left[\Delta X_T\Delta\ii_T^g\right].
		\end{split}
	\end{equation}
	Plugging \eqref{main1_lem1_8'}-\eqref{main1_lem1_5} into \eqref{main2_lem1_11}, we have for any $\epsilon\in(0,1)$,
	\begin{align*}
		&2\gamma C_f\e\left[\int_0^T|\Delta u_t|^2dt\right]\\
		\le\ & \gamma L_m(T+1)\e\left[\sup_{0\le t\le T}|(\Delta X_t,\Delta p_t)|^2+\int_0^T|(\Delta q_t,\Delta\tilde{q}_t)|^2dt\right]\\
		&+\e\bigg[\gamma l_{x^0} |\Delta X_T^0||\Delta X_T|+|\Delta X_T||\Delta\ii_T^g|+|\Delta p_0||\Delta\xi|\\
		&\qquad +\gamma l_{x^0}\int_0^T |\Delta X^0_t|(|\Delta u_t|+2|\Delta X_t|) dt\\
		&\qquad +\int_0^T\left( |\Delta p_t||\Delta\ii_t^b|+|\Delta q_t||\Delta\ii_t^{\sigma}|+|\Delta\tilde{q}_t||\Delta\ii_t^{\tilde{\sigma}}|+|\Delta X_t||\Delta\ii_t^f| \right) dt\bigg]\\
		\le\ & (L_m+\epsilon)(T+1)\e\left[\sup_{0\le t\le T}|(\Delta X_t,\Delta p_t)|^2+\int_0^T|(\Delta q_t,\Delta\tilde{q}_t)|^2dt\right]\\
		& +\epsilon\gamma\e\left[\int_0^T|\Delta u_t|^2dt\right]+\frac{1}{\epsilon}\e\left[3l^2_{x^0}(T+1)\sup_{0\le t\le T}|\Delta X_t^0|^2+|\Delta\xi|^2\right]+\frac{1}{\epsilon}\|\Delta\ii\|^2
	\end{align*}
	Plugging \eqref{main2_lem1_10.5} into above, we have
	\begin{align*}
		2\gamma C_f\e\left[\int_0^T|\Delta u_t|^2dt\right]\le \ & (L_m+\epsilon)\gamma(T+1)(C_2+1)\e\left[ \int_0^T|\Delta u_t|^2 \right]\\
		& +l^2_{x^0}(T+1)\left(\frac{3}{\epsilon}+(L+1)C_2\right)\e\left[\sup_{0\le t\le T}|\Delta X_t^0|^2\right]\\
		& +\left(\frac{1}{\epsilon}+(L+1)(T+1)C_2\right)\left(\e[|\Delta\xi|^2]+\|\Delta\ii\|^2\right).
	\end{align*}
	We now set $\delta_1:=2^{-1}(T+1)^{-1}(C_2+1)^{-1}$, which depends only on $(L,T)$. When $L_mC_f^{-1}\le\delta_1$, we choose $\epsilon=2^{-1}(T+1)^{-1}(C_2+1)^{-1}C_f$, then we have
	\begin{equation}\label{main2_lem1_15}
		\begin{split}
			\gamma \e\left[\int_0^T|\Delta u_t|^2dt\right]\le\ & 6l^2_{x^0}(L+1)^2(T+1)^2(2C_2+1)\e\left[\sup_{0\le t\le T}|\Delta X_t^0|^2\right]\\
			&+2(L+1)^2(T+1)(2C_2+1)\left(\e[|\Delta\xi|^2]+\|\Delta\ii\|^2\right).
		\end{split}
	\end{equation}
	Plugging \eqref{main2_lem1_15} into \eqref{main2_lem1_10.5}, we have
	\begin{align*}
		&\e\left[\sup_{0\le t\le T}|(\Delta X_t,\Delta p_t)|^2+\int_0^T |(\Delta q_t,\Delta \tilde{q}_t)|^2 dt\right]\\
		\le\ & 12 l_{x^0}^2 (L+1)^2(T+1)^2(C_2+1)^2\e\left[ \sup_{0\le t\le T}|\Delta X^0_t|^2\right]\\
		& +C(L,T)\left(\e[|\Delta\xi|^2]+\|\Delta\ii\|^2\right).
	\end{align*}
	We know from Lemma~\ref{mm:lemma:u} that
	\begin{equation*}
		|\Delta u_t|\le \frac{L^2}{2}\left(|\Delta X_t|+|\Delta p_t|+|\Delta q_t|+|\Delta\tilde{q}_t|\right)+\frac{l_{x^0}L}{2} |\Delta X_t^0|,
	\end{equation*}
	so we eventually have 
	\begin{equation}\label{main2_lem1_b}
		\begin{split}
			\|\Delta\Theta\|^2\le  &24 l_{x^0}^2 (L+1)^4(T+1)^3(C_2+1)^2\e\left[ \sup_{0\le t\le T}|\Delta X^0_t|^2\right]\\
			&+C(L,T)\left(\e[|\Delta\xi|^2]+\|\Delta\ii\|^2\right).
		\end{split}
	\end{equation}
	
	In view of \eqref{main2_lem1_a} and \eqref{main2_lem1_b}, we have
	\begin{equation*}
		\begin{split}
			\|(\Delta\Theta^0,\Delta\Theta)\|^2_{\mathbb{S}}\le  & 144 l_{x^0}^2 l_m^2 (L+1)^{10}(T+1)^6 (C_1+1)^2 (C_2+1)^2 \|(\Delta\Theta^0,\Delta\Theta)\|^2_{\mathbb{S}} \\
			&+C(L,T)\left(\e[|\Delta\xi^0|^2+|\Delta\xi|^2]+\|(\Delta\ii^0,\Delta\ii)\|_{\iii}^2\right).
		\end{split}
	\end{equation*}
	Now we set $\delta_2:=24^{-1} (L+1)^{-5}(T+1)^{-3} (C_1+1)^{-1} (C_2+1)^{-1}$, which depends only on $(L,T)$. When $l_{x^0}l_m\le\delta_2$, we have
	\begin{equation}\label{prop_0}
		\begin{split}
			\|(\Delta\Theta^0,\Delta\Theta)\|^2_{\mathbb{S}}\le  C(L,T)\left(\e[|\Delta\xi^0|^2+|\Delta\xi|^2]+\|(\Delta\ii^0,\Delta\ii)\|_{\iii}^2\right).
		\end{split}
	\end{equation}
\end{proof}

We now have the solvability of FBSDEs \eqref{fbsde2}. Here, we also give boundedness of the solution of FBSDEs \eqref{fbsde2}. The proof of the following result is similar as that of \eqref{prop_0}, which is omitted here.
\begin{proposition}\label{prop}
	Let Assumptions (H1)-(H4) be satisfied. Then, there exist $\delta>0$ depending only on $(L,T)$, such that when $L_mC_f^{-1}\vee l_{x^0}l_m\le\delta$, the unique solution $(\Theta^{0},\Theta)$ of FBSDEs \eqref{fbsde2} satisfies	
	\begin{equation*}
		\begin{split}
			\|(\Theta^0,\Theta)\|^2_{\mathbb{S}}\le  C\left(1+\e[|\xi^0|^2+|\xi|^2]\right)
		\end{split}
	\end{equation*}
	for a positive constant $C$ depending only on $(L,T)$.
\end{proposition}
%%%%%%%%%%%%%%%%%%%%%%%%%%%%%%%%%%%%%%%%%%%%%%%%%%%%%%%%%%%%%%%%%%%%%%%%%%%%
\section{Solvability of Problem~\ref{def1}}\label{sec_nash}
In this section, we show how the solution of Problem~\ref{def2} can provide an $\mathcal{O}(\frac{1}{\sqrt{N}})$-Nash equilibrium for Problem~\ref{def1}. We always suppose that Assumptions (H1)-(H5) hold true. 

For the Brownian motion $(W^0,W)$ and the initial data $(\xi^0,\xi)$, we denote by
\begin{equation}\label{solution}
	\left(\bar{X}^0,\bar{p}^0;\bar{q}^0;\bar{X},\bar{p};\bar{q},\bar{\tilde{q}}\right)\in(\sr^2_{\f^0})^2\times\lr^2_{\f^0}\times(\sr^2_{\f})^2\times(\lr^2_{\f})^2 (0,T)
\end{equation}
the solution of FBSDEs \eqref{fbsde2}, and set $m_t:=\lr(\bar{X}_t|\f_t^0),\ 0\le t\le T$. Here, $\f^0$ is the natural filtration of $(\xi^0,W^0)$ and $\f$ is the natural filtration of $(\xi^0,\xi,W^0,W)$. Given $(\boldsymbol{\xi},\boldsymbol{W})$ independent of $(\xi^0,W^0)$ and $\{(\bar{X}_t^0,m_t),\ 0\le t\le T \}$ defined as above, we consider the following FBSDEs for $(\boldsymbol{X},\boldsymbol{p},\boldsymbol{q},\boldsymbol{\tilde{q}})$:
\begin{equation}\label{fbsde3}
	\left\{
	\begin{aligned}
		&d\boldsymbol{X}_t=b(t,\boldsymbol{X}_t,\hat{u}(t,\boldsymbol{X}_t,\boldsymbol{p}_t,\boldsymbol{q}_t,\boldsymbol{\tilde{q}}_t,\bar{X}^0_t),m_t)dt\\
		&\qquad\quad+\sigma(t,\boldsymbol{X}_t,\hat{u}(t,\boldsymbol{X}_t,\boldsymbol{p}_t,\boldsymbol{q}_t,\boldsymbol{\tilde{q}}_t,\bar{X}^0_t),m_t)d\boldsymbol{W}_t\\
		&\qquad\quad+\tilde{\sigma}(t,\boldsymbol{X}_t,\hat{u}(t,\boldsymbol{X}_t,\boldsymbol{p}_t,\boldsymbol{q}_t,\boldsymbol{\tilde{q}}_t,\bar{X}^0_t),m_t)d{W}^0_t,\quad t\in(0,T];\\
		&d\boldsymbol{p}_t=- H_x(t,\boldsymbol{X}_t,\boldsymbol{p}_t,\boldsymbol{q}_t,\boldsymbol{\tilde{q}}_t,\hat{u}(t,\boldsymbol{X}_t,\boldsymbol{p}_t,\boldsymbol{q}_t,\boldsymbol{\tilde{q}}_t,\bar{X}^0_t),m_t,{X}^0_t)dt\\
		&\qquad\quad+\boldsymbol{q}_td\boldsymbol{W}_t+\boldsymbol{\tilde{q}}_td{W}^0_t,\quad t\in[0,T);\\
		&\boldsymbol{X}_0=\boldsymbol{\xi},\  \boldsymbol{p}_T=g_x(\boldsymbol{X}_T,m_T,\bar{X}^0_T).
	\end{aligned}
	\right.
\end{equation}
As a direct consequence of Theorem~\ref{main2_thm}, FBSDEs \eqref{fbsde3} has a unique solution $(\boldsymbol{X},\boldsymbol{p},\boldsymbol{q},\boldsymbol{\tilde{q}})$, which is, in the sense of  Carmona and Delarue \cite[Definition 1.17]{book_mfg}, a solution  of FBSDEs \eqref{fbsde3}  in the random environment $\boldsymbol{\mathcal{H}}:=(W^0,\bar{X}^0,m,\boldsymbol{\xi},\boldsymbol{W})$. In the  random environment $\mathcal{H}:=(W^0,\bar{X}^0,m,{\xi},{W})$, $(\bar{X},\bar{p},\bar{q},\bar{\tilde{q}})$  is the unique solution of the underlying FBSDEs. In Problem~\ref{def1},   $(W^1,\ldots, W^N)$ is an $N$-dimensional Brownian motion, and the random variables $\xi^i,\ 1\le i\le N,$ are  independent copies of the random variable $\xi$  and are independent of $(W^1,\ldots, W^N)$. For $1\le i\le N$, $(\xi^i, W^i)$ are independent of $(\xi^0,W^0)$, and $\f^i$ is the natural filtration of $(\xi^0,\xi^i,W^0,W^i)$  augmented by all the $\mathbb{P}$-null sets. We denote by $(\bar{X}^i,\bar{p}^i,\bar{q}^i,\bar{\tilde{q}}^i)\in (\sr^2_{\f^i}(0,T))^2\times(\lr^2_{\f^i}(0,T))^2$ the solution of \eqref{fbsde3} in the random environment $\mathcal{H}^i:=(W^0,\bar{X}^0,m,\xi^i,W^i)$. In view of~\cite[Theorem 1.33]{book_mfg} and the identical distributions of the $N$ random variables $\xi^i,\ 1\le i\le N$,  we have 
\begin{equation*}
	\lr(\bar{X}_t^i|\f^0_t)=\lr(\bar{X}_t|\f^0_t)=m_t, \ t\in[0,T],\   \text{a.s.},\ 1\le j\le N.
\end{equation*}
Therefore, for $1\le i\le N$,
\begin{equation*}
	(\bar{X}^0,\bar{p}^0,\bar{q}^0,\bar{X}^i,\bar{p}^i,\bar{q}^i,\bar{\tilde{q}}^i)\in(\sr^2_{\f^0})^2\times\lr^2_{\f^0}\times(\sr^2_{\f^i})^2\times(\lr^2_{\f^i})^2(0,T)
\end{equation*}
is actually the solution of FBSDEs \eqref{fbsde2} with the Brownian motion and the initial $(W, \xi):=(W^i, \xi^i)$, and for $1\le i,j\le N$, we have $\lr(\bar{X}_t^i|\f^0_t)=\lr(\bar{X}_t^j|\f^0_t)=m_t$. We set
\begin{align*}
	&\bar{u}^0_t:=\hat{u}^0(t,\bar{X}^0_t,\bar{p}^0_t,\bar{q}^0_t,m_t),\\
	&\bar{u}^i_t:=\hat{u}(t,\bar{X}^i_t,\bar{p}^i_t,\bar{q}^i_t,\bar{\tilde{q}}^i_t,\bar{X}^0_t),\quad 1\le i\le N.
\end{align*}
From Theorem~\ref{thm:smp},  we know that $(\bar{u}^0,\bar{u}^i)$ is a solution of Problem~\ref{def2} with $(W,\xi):=(W^i,\xi^i)$ for $1\le i\le N$. 

\begin{remark}
	We refer to Nourian and  Caines~\cite{NM} for a discussion on the conditional distribution of $\{\bar X^i,\ 1\le i\le N\}$ under a different setting. We also refer to  Carmona and Delarue \cite[Chapter 1]{book_mfg} for more details about FBSDEs in a random environment. 
	
	Since $m_t=\lr(\bar{X}_t|\f_t^0)$, we can see that FBSDEs \eqref{fbsde3} in the random environment $\mathcal{H}$ is a system of conditional distribution dependent FBSDEs for $(\bar{X},\bar{p},\bar{q},\bar{\tilde{q}})$. Following the method of our previous work \cite[Section 5]{HT} and Ahuja et al. \cite[Section 3.3]{SAWR}, we know that there is a deterministic function $U:[0,T]\times\br\times\pr_2(\br)\to\br$ (which also depends on $\bar{X}^0$), such that	
	\begin{equation*}
		\bar{p}_t=U(t,\bar{X}_t,\lr(\bar{X}_t|\f^0_t)),  \quad t\in[0,T],\   \text{a.s.}
	\end{equation*}
	Furthermore, if the diffusion $(\sigma,\tilde{\sigma})$ does not depend on the control $u$, then, the map $\hat{u}(\cdot)$ does not depend on $(q,\tilde{q})$. Therefore, $\bar{X}$ is actually the solution of the following McKean-Vlasov SDE:
	\begin{equation}\label{MV}
		\begin{split}
			\bar{X}_t=\xi&+\int_0^t b\left(s,\bar{X}_s,\hat{u}\left(s,\bar{X}_s,U(t,\bar{X}_s,\lr(\bar{X}_s|\f^0_s)),\bar{X}^0_s\right),\lr(\bar{X}_s|\f_t^0)\right)ds\\
			&+\int_0^t \sigma \left(s,\bar{X}_s,\hat{u}\left(s,\bar{X}_s,U(t,\bar{X}_s,\lr(\bar{X}_s|\f^0_s)),\bar{X}^0_s\right),\lr(\bar{X}_s|\f_t^0)\right)dW_s\\
			&+\int_0^t \tilde{\sigma} \left(s,\bar{X}_s,\hat{u}\left(s,\bar{X}_s,U(t,\bar{X}_s,\lr(\bar{X}_s|\f^0_s)),\bar{X}^0_s\right),\lr(\bar{X}_s|\f_t^0)\right)dW^0_s,
		\end{split}
	\end{equation}
	and $\{\bar{X}^i,\ 1\le i\le N\}$ are samples of the McKean-Vlasov SDE \eqref{MV}. We also refer to Carmona and Delarue \cite[Chapter 2.1]{book_mfg} for a further discussion.
\end{remark}

In view of Assumption (H5), we  use the notations
\begin{align*}
	&b_0^0(t,m)=\int_\br \overline{b_0^0}(t,y)m(dy),\quad \sigma_0^0(t,m)=\int_\br \overline{\sigma_0^0}(t,y)m(dy),\\
	&b_0(t,m)=\int_\br \overline{b_0}(t,y)m(dy),\quad \sigma_0(t,m)=\int_\br \overline{\sigma_0}(t,y)m(dy),\quad \tilde{\sigma}_0(t,m)=\int_\br \overline{\tilde{\sigma}_0}(t,y)m(dy),\\
	&f^0(t,x^0,u^0,m)=\int_\br \overline{f^0}(t,x^0,u^0,y)m(dy),\quad g^0(x^0,m)=\int_\br \overline{g^0} (x^0,y)m(dy),\\
	&f(t,x,u,m,x^0)=\int_\br \overline{f}(t,x,u,y,x^0)m(dy),\quad g(x,m,x^0)=\int_\br \overline{g}(x,y,x^0)m(dy),
\end{align*}
with functions $(\overline{b_0^0},\overline{\sigma_0^0},\overline{b_0},\overline{\sigma_0},\overline{\tilde{\sigma}_0},\overline{f^0}, \overline{g^0}, \overline{f},\overline{g})$ being $L$-Lipschitz contimuous in $y\in\br$. Applying $(\bar{u}^0,\dots,\bar{u}^N)$ into the $(N+1)$-agent game, the dynamics of agents are as follows: for $1\le i\le N$,
\begin{equation}\label{def:X^iN}
	\begin{split}
		&\bar{X}^{0,N}_t=\xi^0+\int_0^t \bigg[\frac{1}{N}\sum_{j=1}^N \overline{b_0^0}(s,\bar{X}^{j,N}_s)+b_1^0(s)\bar{X}^{0,N}_s+b_2^0(s)\bar{u}_s^0\bigg]ds\\
		&\qquad\qquad +\int_0^t \bigg[\frac{1}{N}\sum_{j=1}^N \overline{\sigma_0^0}(s,\bar{X}^{j,N}_s)+\sigma_1^0(s)\bar{X}^{0,N}_s+\sigma_2^0(s)\bar{u}_s^0 \bigg]dW_s^0, \quad t\in[0,T],\\
		&\bar{X}^{i,N}_t=\xi^i+\int_0^t \bigg[\frac{1}{N}\sum_{j=1}^N \overline{b_0}(s,\bar{X}^{j,N}_s)+b_1(s)\bar{X}^{i,N}_s+b_2(s)\bar{u}_s^i \bigg]ds\\
		&\qquad\qquad+\int_0^t \bigg[\frac{1}{N}\sum_{j=1}^N \overline{\sigma_0}(s,\bar{X}^{j,N}_s)+\sigma_1(s)\bar{X}^{i,N}_s+\sigma_2(s)\bar{u}_s^i \bigg] dW_s^i\\
		&\quad\quad\quad\quad +\int_0^t \bigg[\frac{1}{N}\sum_{j=1}^N \overline{\tilde{\sigma}_0}(s,\bar{X}^{j,N}_s)+\tilde{\sigma}_1(s)\bar{X}^{i,N}_s+\tilde{\sigma}_2(t)\bar{u}_s^i \bigg] dW_s^0,\quad t\in[0,T].
	\end{split}
\end{equation}
The following lemma gives estimate of the distance between $\bar{X}^i$ and $\bar{X}^{i,N}$.

\begin{lemma}\label{epsilon_lemma}
	Let Assumptions (H1)-(H5) be satisfied. Then, we have for $0\le i\le N$,
	\begin{equation}\label{epsilon_lem_1}
		\sup_{0\le t\le T}\e\left[|\bar{X}_t^{i,N}-\bar{X}_t^{i}|^2\right]\le \frac{C}{N}\left(1+\e[|\xi^0|^2+|\xi^1|^2]\right)
	\end{equation}
	for a positive constant $C$ depending only on $(L,T)$.
\end{lemma}
\begin{proof}
	We set $\Delta X_t^i=\bar{X}_t^{i,N}-\bar{X}_t^{i}$ for $0\le i\le N$, then we have
	\begin{align*}
		&d\Delta X_t^0=\bigg[\bigg(\frac{1}{N}\sum_{j=1}^N \overline{b_0^0}(t,\bar{X}^{j,N}_t)-b_0^0(t,m_t)\bigg)+b_1^0(t)\Delta X_t^0\bigg]dt\\
		&\qquad\qquad+\bigg[\bigg(\frac{1}{N}\sum_{j=1}^N \overline{\sigma_0^0}(t,\bar{X}^{j,N}_t)-\sigma_0^0(t,m_t)\bigg)+\sigma_1^0(t)\Delta X_t^0 \bigg]dW_t^0,\quad t\in(0,T],\\
		&d\Delta X_t^i=\bigg[\bigg(\frac{1}{N}\sum_{j=1}^N \overline{b_0}(t,\bar{X}^{j,N}_t)-b_0(t,m_t) \bigg)+b_1(t)\Delta X_t^i \bigg]dt\\
		&\qquad\qquad +\bigg[\bigg(\frac{1}{N}\sum_{j=1}^N \overline{\sigma_0}(t,\bar{X}^{j,N}_t)-\sigma_0(t,m_t) \bigg)+\sigma_1(t)\Delta X_t^i \bigg]dW_t^i\\
		&\quad\quad\quad\quad +\bigg[\bigg(\frac{1}{N}\sum_{j=1}^N \overline{\tilde{\sigma}_0}(t,\bar{X}^{j,N}_t)-\tilde{\sigma}_0(t,m_t)\bigg)+\tilde{\sigma}_1(t)\Delta X_t^i\bigg]dW_t^0,\quad t\in(0,T], 
	\end{align*}
	where $m_t=\lr(\bar{X}^{i}_t|\f_t^0)$ for all $1\le i\le N$. From Assumption (H5), we have
	\begin{equation*}
		\bigg|\frac{1}{N}\sum_{j=1}^N \overline{b_0^0}(t,\bar{X}^{j,N}_t)-b_0^0(t,m_t)\bigg|\le \frac{L}{N}\sum_{j=1}^N|\Delta X_t^j|+\bigg|\frac{1}{N}\sum_{j=1}^N \overline{b_0^0}(t,\bar{X}^{j}_t)-b_0^0(t,m_t)\bigg|.
	\end{equation*}    
	Therefore, from standard estimates for SDEs, we have
	\begin{equation}\label{6}
		\begin{split}
			&\e\Big[\sup_{0\le t\le T}|\Delta X_t^0|^2\Big]\\
			\le\ & C\bigg(\e\bigg[\frac{1}{N}\sum_{j=1}^N\int_0^T|\Delta X_t^j|^2dt\bigg]+\e\bigg[\int_0^T \bigg|\frac{1}{N}\sum_{j=1}^N \overline{b_0^0}(t,\bar{X}^{j}_t)-b_0^0(t,m_t)\bigg|^2 dt\bigg]\\
			&\qquad +\e\bigg[\int_0^T \bigg|\frac{1}{N}\sum_{j=1}^N \overline{\sigma_0^0}(t,\bar{X}^{j}_t)-\sigma_0^0(t,m_t)\bigg|^2 dt\bigg]\bigg),
		\end{split}
	\end{equation}
	for a positive constant $C$ depending only on $(L,T)$. In the rest of the proof, $C$ always stands for a constant depending only on $(L,T)$, which may vary from line to line. Similarly, we have
	\begin{equation}\label{7}
		\begin{split}
			&\e\Big[\sup_{0\le t\le T}|\Delta X_t^i|^2\Big]\\
			\le\ &  C\bigg(\e\bigg[\frac{1}{N}\sum_{j=1}^N\int_0^T|\Delta X_t^j|^2dt\bigg]+\e\bigg[\int_0^T \bigg|\frac{1}{N}\sum_{j=1}^N \overline{b_0^0}(t,\bar{X}^{j}_t)-b_0^0(t,m_t)\bigg|^2dt\bigg]\\
			&\qquad +\e\bigg[\int_0^T \bigg|\frac{1}{N}\sum_{j=1}^N \overline{\sigma_0^0}(t,\bar{X}^{j}_t)-\sigma_0^0(t,m_t)\bigg|^2dt\bigg]\bigg).
		\end{split}
	\end{equation}    
	From \eqref{6}, \eqref{7} and Gronwall's inequality, we have
	\begin{equation}\label{8}
		\begin{split}
			&\e\Big[\sup_{0\le t\le T}|\Delta X_t^0|^2\Big]+\frac{1}{N}\sum_{j=1}^N\e\Big[\sup_{0\le t\le T}|\Delta X_t^j|^2\Big]\\
			\le\ & C\e\bigg[\int_0^T\bigg( \bigg|\frac{1}{N}\sum_{j=1}^N \overline{b_0^0}(t,\bar{X}^{j}_t)-b_0^0(t,m_t)\bigg|^2 + \bigg|\frac{1}{N}\sum_{j=1}^N \overline{\sigma_0^0}(t,\bar{X}^{j}_t)-\sigma_0^0(t,m_t)\bigg|^2 \bigg)dt\bigg].
		\end{split}
	\end{equation}
	We set $Y_t^j:=\overline{b_0^0}(t,\bar{X}^{j}_t)-b_0^0(t,m_t)$, then, we know from the definition of $m_t$ that $Y_t^j=\overline{b_0^0}(t,\bar{X}^{j}_t)-\e[\overline{b_0^0}(t,\bar{X}^{j}_t)|\f_t^0]$ and $\e[Y_t^j|\f_t^0]=0$. Recall that $W^j$ and $W^{j'}$ are independent when $1\le j\neq j'\le N$, so we have $\e[Y_t^j Y_t^{j'}|\f_t^0]=\e[Y_t^j|\f_t^0]\e[Y_t^{j'}|\f_t^0]=0$, and then $\e[Y_t^j Y_t^{j'}]=\e\big[\e[Y_t^j Y_t^{j'}|\f_t^0]\big]=0$. Since $\{Y_t^j,1\le j\le N\}$ are identically distributed, we have 
	\begin{equation}\label{5}
		\begin{split}
			\e\bigg[\bigg|\frac{1}{N}\sum_{j=1}^N \overline{b_0^0}(t,\bar{X}^{j}_t)-b_0^0(t,m_t)\bigg|^2\bigg]&=\frac{1}{N^2}\sum_{j=1}^N\e\left[|Y_t^j|^2\right]+\frac{1}{N^2}\sum_{1\le j\neq j'\le N}\e\left[Y_t^jY_t^{j'}\right]\\
			&=\frac{1}{N^2}\sum_{j=1}^N\e\left[|Y_t^j|^2\right].
		\end{split}
	\end{equation}
	Here, from Assumptions (H1) and (H5), the fact that $\{\bar{X}^{j}_t,\ 1\le j\le N\}$ are identically distributed, and standard estimates for SDEs, we have
	\begin{equation}\label{5_1}
		\begin{split}
			\frac{1}{N^2}\sum_{j=1}^N\e\left[|Y_t^j|^2\right]&=\frac{1}{N}\e\left[|Y_t^1|^2\right]=\frac{1}{N}\e\left[\left|\overline{b_0^0}(t,\bar{X}^{1}_t)-b_0^0(t,m_t)\right|^2\right]\\
			&\le \frac{8L^2}{N}\left(1+\e[|\bar{X}^{1}_t|^2]\right)\\
			&\le \frac{C}{N}\left(1+\e\left[|\xi^1|^2+\int_0^T |\bar{u}^{1}_t|^2dt\right]\right).
		\end{split}
	\end{equation}
	From \eqref{5} and \eqref{5_1}, we have    
	\begin{equation}\label{5_2}
		\begin{split}
			\e\bigg[\bigg|\frac{1}{N}\sum_{j=1}^N \overline{b_0^0}(t,\bar{X}^{j}_t)-b_0^0(t,m_t)\bigg|^2\bigg]& \le \frac{C}{N}\left(1+\e[|\xi^1|^2+\int_0^T |\bar{u}^{1}_t|^2dt]\right).
		\end{split}
	\end{equation}
	Similarly, we have
	\begin{equation}\label{5'}
		\begin{split}
			\e\bigg[\bigg|\frac{1}{N}\sum_{j=1}^N \overline{\sigma_0^0}(t,\bar{X}^{j}_t)-\sigma_0^0(t,m_t)\bigg|^2\bigg] \le \frac{C}{N}\left(1+\e[|\xi^1|^2+\int_0^T |\bar{u}^{1}_t|^2dt]\right).
		\end{split}
	\end{equation} 
	Plugging \eqref{5_2} and \eqref{5'} into \eqref{8}, 
	\begin{equation}\label{5''}
		\sup_{0\le t\le T}\e\left[\left|\bar{X}_t^{i,N}-\bar{X}_t^{i}\right|^2\right] \le \frac{C}{N}\left(1+\e\left[|\xi^1|^2+\int_0^T |\bar{u}^{1}_t|^2dt\right]\right).
	\end{equation}	
	Recall that $(\bar{X}^0,\bar{p}^0,\bar{q}^0,\bar{X}^1,\bar{p}^1,\bar{q}^1,\bar{\tilde{q}}^1)$ is the solution of FBSDEs \eqref{fbsde2} with the Brownian motion $W:=W^1$ and the initial $\xi:=\xi^1$. Therefore, from Proposition~\ref{prop}, we have
	\begin{equation}\label{boundedness}
		\begin{split}
			\e\left[\int_0^T(|\bar{u}^{0}_t|^2+|\bar{u}^{1}_t|^2) dt \right]\le C\left(1+\e[|\xi^0|^2+|\xi^1|^2]\right).
		\end{split}
	\end{equation}
	Plugging \eqref{boundedness} into \eqref{5''}, we obtain \eqref{epsilon_lem_1}.	
\end{proof}

Our main result of this section is stated as follows.

\begin{theorem}\label{thm:nash}
	Let Assumptions (H1)-(H5) be satisfied. Then, $(\bar{u}^0,\dots,\bar{u}^N)$ is a  $\frac{C}{\sqrt{N}}$-Nash equilibrium of the $(N+1)$-agent game, where $C$ is a positive constant depending only on $(L,T,\e[|\xi^0|^2],\e[|\xi^1|^2])$. That is, $(\bar{u}^0,\dots,\bar{u}^N)$ is an $\mathcal{O}(\frac{1}{\sqrt{N}})$-Nash equilibrium of the $(N+1)$-agent game.
\end{theorem}
\begin{proof}
	\textbf{Case 1 (strategy change of the major agent)}:\\
	While the minor agents are using the responses $(\bar{u}^1,\dots,\bar{u}^N)$, the change from strategy $\bar{u}^0$ to some $u^0\in\lr_{\f^0}^2(0,T)$ of the major agent yields
	\begin{equation*}
		\begin{split}
			&X^{0}_t=\xi^0+\int_0^t \bigg[ \frac{1}{N}\sum_{j=1}^N \overline{b_0^0}(s,\bar{X}^{j,N}_s)+b_1^0(s){X}^{0}_s+b_2^0(s){u}^0_s \bigg]ds\\
			&\qquad\qquad+\int_0^t \bigg[ \frac{1}{N}\sum_{j=1}^N \overline{\sigma_0^0}(s,\bar{X}^{j,N}_s)+\sigma_1^0(s){X}^{0}_s+\sigma_2^0(s){u}^0_s \bigg] dW_s^0, \quad t\in[0,T].
		\end{split}
	\end{equation*}
	Note that the $i$-th minor agent's state is still $\bar{X}^{i,N}$ for $1\le i\le N$. We define
	\begin{equation*}
		\begin{split}
			&\hat{X}^0_t:=\xi^0+\int_0^t \bigg[ b_0^0(s,m_s)+b_1^0(s)\hat{X}^0_s+b_2^0(s){u}^0_s \bigg] ds\\
			&\qquad\qquad+\int_0^t \bigg[\sigma_0^0(s,m_s)+\sigma_1^0(s)\hat{X}^0_s+\sigma_2^0(s){u}^0_s \bigg] dW_s^0,\quad t\in[0,T].
		\end{split}
	\end{equation*}
	Similar to the proof Lemma~\ref{epsilon_lemma}, since the drift and diffusion functions are both linear in control and state, we have
	\begin{align}\label{v1}
		&\e\Big[\sup_{0\le t\le T}|{X}_t^{0}-\hat{X}_t^{0}|^2\Big] \le \frac{C_1}{N}\left(1+\e\left[|\xi^1|^2+\int_0^T |\bar{u}^{1}_t|^2dt\right]\right).
	\end{align}
	Here and in the rest of the proof, $C_1$ always stands for a constant depending only on $(L,T)$, which may vary from line to line. Recall that
	\begin{equation}\label{v_2}
		\begin{split}
			J_0(u^0|\bar{u}^{-0})=&\e\bigg[\int_0^T\frac{1}{N}\sum_{j=1}^N \overline{f^0}(t,{X}_t^0,u^0_t,\bar{X}_t^{j,N})dt+\frac{1}{N}\sum_{j=1}^N \overline{g^0}({X}_T^0,\bar{X}_T^{j,N})\bigg].
		\end{split}
	\end{equation}    
	From Assumption (H2), standard estimates for SDEs, Lemma~\ref{epsilon_lemma}, \eqref{boundedness} and \eqref{v1}, we have for any $1\le j\le N$,   
	\begin{align}
		&\e\left[\int_0^T\left|\overline{f^0}(t,{X}_t^0,u^0_t,\bar{X}_t^{j,N})-\overline{f^0}(t,\hat{X}_t^0,u^0_t,\bar{X}_t^j)\right|dt\right] \notag\\
		\le\ & 2L\e\bigg[\int_0^T\left(1+|{X}_t^0|+ |\hat{X}_t^0|+|u^0_t|+|\bar{X}_t^{j,N}|+|\bar{X}_t^{j}|\right) \notag\\
		&\qquad\qquad\quad \cdot  \left(|{X}_t^0-\hat{X}_t^0|+|\bar{X}_t^{j,N}-\bar{X}_t^{j}|\right) dt\bigg] \notag\\
		\le\ & C_1\left(1+\e\left[|\xi^0|^2+|\xi^1|^2+\int_0^T \left(|u^0_t|^2+|\bar{u}_t^{0}|^2+|\bar{u}_t^{1}|^2\right) dt\right]\right)^{\frac{1}{2}} \notag\\
		&\qquad \cdot \left( \e\bigg[\sup_{0\le t\le T}|{X}_t^{0}-\hat{X}_t^{0}|^2+\sup_{0\le t\le T}|\bar{X}_t^{j,N}-\bar{X}_t^{j}|^2\bigg]\right)^\frac{1}{2} \notag\\
		\le\ & \frac{C_1}{\sqrt{N}}\left(1+\e\left[|\xi^0|^2+|\xi^1|^2+\int_0^T \left(|u^0_t|^2+|\bar{u}_t^{0}|^2+|\bar{u}_t^{1}|^2\right) dt\right]\right)^{\frac{1}{2}} \notag\\
		&\qquad \cdot \left( 1+\e\left[|\xi^1|^2+\int_0^T |\bar{u}^{1}_t|^2dt\right]\right)^\frac{1}{2} \notag\\
		\le\ & \frac{C}{\sqrt{N}} \left(1+\e\left[\int_0^T |u^0_t|^2 dt\right]\right).\label{v_3}
	\end{align}
	Here and  in the rest of the proof,  $C$ is a positive constant depending only on the quantities
	\begin{align}\label{v_5}
		L,\ T,\ \e[|\xi^0|^2],\ \e[|\xi^1|^2].
	\end{align} 
	Plugging \eqref{v_3} (and a similar inequality for $\overline{g^0}$) into \eqref{v_2}, we have
	\begin{equation}\label{v_4}
		\begin{split}
			J_0(u^0|\bar{u}^{-0})\geq\ &\e\bigg[\int_0^T\frac{1}{N}\sum_{j=1}^N \overline{f^0}(t,\hat{X}_t^0,u^0_t,\bar{X}_t^j)dt+\frac{1}{N}\sum_{j=1}^N \overline{g^0}(\hat{X}_T^0,\bar{X}_T^j)\bigg]\\
			&-\frac{C}{\sqrt{N}}\left(1+\e\left[\int_0^T|u_t^0|^2 dt\right]\right).
		\end{split}
	\end{equation} 
	Similar to the proof of \eqref{5_2}, we have
	\begin{equation}\label{3}
		\begin{split}
			&\e\bigg[\int_0^T \bigg|\frac{1}{N}\sum_{j=1}^N(\overline{f^0}(t,\hat{X}_t^0,u^0_t,\bar{X}_t^j)-f^0(t,\hat{X}_t^0,u^0_t,m_t))\bigg|^2dt\\
			&\quad +\bigg|\frac{1}{N}\sum_{j=1}^N(\overline{g^0}(\hat{X}^0_T,\bar{X}_T^j)-g^0(\hat{X}^0_T,m_T))\bigg|^2\bigg]\le \frac{C}{N} \bigg(1+\e\bigg[\int_0^T|u_t^0|^2 dt\bigg]^2\bigg).
		\end{split}
	\end{equation}
	Plugging \eqref{3} into \eqref{v_4}, we have
	\begin{equation}\label{epsilon_thm_3}
		\begin{split}
			J_0(u^0|\bar{u}^{-0})\geq\ &\e\left[\int_0^T f^0(t,\hat{X}_t^0,u^0_t,m_t)dt+ g^0(\hat{X}_T^0,m_T)\right]\\
			& -\frac{C}{\sqrt{N}}\e\left[\int_0^T|u_t^0|^2 dt\right]-\frac{C}{\sqrt{N}}\\
			=\ & J_0(u^0|m)-\frac{C}{\sqrt{N}}\e\left[\int_0^T|u_t^0|^2 dt\right]-\frac{C}{\sqrt{N}}.
		\end{split}
	\end{equation}
	In a similar way, we have
	\begin{equation}\label{epsilon_thm_4}
		J_0(\bar{u}^0|\bar{u}^{-0})\le J_0(\bar{u}^0|m)+\frac{C}{\sqrt{N}}.
	\end{equation}
	From Theorem~\ref{thm:smp} we know that
	\begin{equation}\label{epsilon_thm_5}
		J_0(\bar{u}^0|m)=\inf_{u^0\in \lr^2_{\f^0}(0,T)}J_0(u^0|m).
	\end{equation}  
	From the following Lemma~\ref{lem:1}, we know that
	\begin{equation}\label{epsilon_thm_5.1}
		\begin{split}
			J_0(u^0|m)-\frac{C}{\sqrt{N}}\e\left[\int_0^T|u_t^0|^2 dt\right]&\geq \inf_{u^0\in \lr^2_{\f^0}(0,T)}\left\{J_0(u^0|m)-\frac{C}{\sqrt{N}}\e\left[\int_0^T|u_t^0|^2 dt\right]\right\}\\
			&\geq \inf_{u^0\in \lr^2_{\f^0}(0,T)}J_0(u^0|m)-\frac{C}{\sqrt{N}}.
		\end{split}
	\end{equation}      
	Therefore, we know from \eqref{epsilon_thm_3}-\eqref{epsilon_thm_5.1} that
	\begin{equation*}\label{epsilon_thm_6}
		J_0(\bar{u}^0|\bar{u}^{-0})-\frac{C}{\sqrt{N}}\le J_0(u^0|\bar{u}^{-0}).
	\end{equation*}
	
	\textbf{Case 2 (strategy change of a minor agent)}:\\
	Without loss of generality, we assume that the $1$-th minor agent changes his/her best response control strategy $\bar{u}^1$ to $u\in\lr_{\f^1}^2(0,T)$. This leads to
	\begin{align*}
		&X^{0}_t=\xi^0+\int_0^t \bigg[\frac{1}{N}\sum_{j=1}^N \overline{b_0^0}(s,{X}^{j}_s)+b_1^0(s){X}^{0}_s+b_2^0(s)\bar{u}^0_s\bigg]ds\\
		&\qquad\qquad +\int_0^t \bigg[\frac{1}{N}\sum_{j=1}^N \overline{\sigma_0^0}(t,{X}^{j}_s)+\sigma_1^0(s){X}^{0}_s+\sigma_2^0(s)\bar{u}^0_s\bigg] dW_s^0,\quad t\in[0,T],\\
		&{X}^{1}_t=\xi^1+\int_0^t \bigg[\frac{1}{N}\sum_{j=1}^N \overline{b_0}(s,{X}^{j}_s)+b_1(s){X}^{1}_s+b_2(s){u}_s\bigg]ds\\
		&\qquad\qquad+\int_0^t \bigg[ \frac{1}{N}\sum_{j=1}^N \overline{\sigma_0}(s,{X}^{j}_s)+\sigma_1(t){X}^{1}_t+\sigma_2(s)u_s\bigg]dW_s^1\\
		&\quad\quad\quad\quad +\int_0^t \bigg[\frac{1}{N}\sum_{j=1}^N \overline{\tilde{\sigma}_0}(s,{X}^{j}_s)+\tilde{\sigma}_1(s){X}^{1}_s+\tilde{\sigma}_2(s){u}_s\bigg]dW_s^0,\quad  t\in[0,T],\\
		&{X}^{i}_t=\xi^i+\int_0^t \bigg[\frac{1}{N}\sum_{j=1}^N \overline{b_0}(s,{X}^{j}_s)+b_1(s){X}^{i}_s+b_2(s)\bar{u}_s^i \bigg]ds\\
		&\quad\quad\quad\quad +\int_0^t \bigg[\frac{1}{N}\sum_{j=1}^N \overline{\sigma_0}(s,{X}^{j}_s)+\sigma_1(s){X}^{i}_s+\sigma_2(s)\bar{u}_s^i\bigg]dW_s^i\\
		&\quad\quad\quad\quad +\int_0^t \bigg[\frac{1}{N}\sum_{j=1}^N \overline{\tilde{\sigma}_0}(s,{X}^{j}_s)+\tilde{\sigma}_1(s)\bar{X}^{i}_s+\tilde{\sigma}_2(s)\bar{u}_s^i\bigg] dW_s^0,\quad t\in[0,T],\quad 2\le i\le N.
	\end{align*}
	We have from standard estimates for SDEs, 
	\begin{equation*}\label{epsilon_thm_7}
		\sup_{j=0,2,3,\dots,N}\e\Big[\sup_{0\le t\le T}|X_t^j-\bar{X}_t^{j,N}|^2\Big] \le \frac{C_1}{N^2}\left(1+\e[|\xi^1|^2+\int_0^T( |{u}_t|^2+|\bar{u}^{1}_t|^2)dt]\right),
	\end{equation*}
	and further from Lemma~\ref{epsilon_lemma}, 
	\begin{equation}\label{epsilon_thm_9}
		\begin{split}
			\sup_{j=0,2,3,\dots,N}\e[\sup_{0\le t\le T}|X_t^{j}-\bar{X}_t^{j}|^2]\le \ &  \frac{C_1}{N}\left(1+\e\left[|\xi^1|^2+\int_0^T( |{u}_t|^2+|\bar{u}^{1}_t|^2)dt\right]\right)\\
			& + \frac{C_1}{N^2}\e\left[\int_0^T |{u}_t|^2dt\right].
		\end{split}
	\end{equation}
	Here and in the rest of the proof, $C_1$ always stands for a constant depending only on $(L,T)$, which may vary from  line to line. We define 
	\begin{equation}\label{1}
		\begin{split}
			&\hat{X}^{1,N}_t:=\xi^1+\int_0^t \bigg[\frac{1}{N}\bigg(\overline{b_0}(s,\hat{X}^{1,N}_s)+\sum_{j=2}^N \overline{b_0}(s,\bar{X}^{j}_s)\bigg)+b_1(s)\hat{X}^{1,N}_s+b_2(s){u}_s\bigg]ds\\
			&\ \ \quad\qquad +\int_0^t\bigg[ \frac{1}{N}\bigg(\overline{\sigma_0}(s,\hat{X}^{1,N}_s)+\sum_{j=2}^N \overline{\sigma_0}(s,\bar{X}^{j}_s)\bigg)+\sigma_1(s)\hat{X}^{1,N}_s+\sigma_2(s){u}_s\bigg]dW_s^1\\
			&\ \ \quad\qquad +\int_0^t\bigg[ \frac{1}{N}\bigg(\overline{\tilde{\sigma}_0}(s,\hat{X}^{1,N}_s)+\sum_{j=2}^N \overline{\tilde{\sigma}_0}(s,\bar{X}^{j}_s)\bigg)+\tilde{\sigma}_1(s)\hat{X}^{1,N}_s+\tilde{\sigma}_2(s){u}_s\bigg]dW_s^0.
		\end{split}		
	\end{equation}  
	Note that \eqref{1} is a SDE for $\hat{X}^{1,N}$ (given $u$ and $\{\bar{X}^j,\ 2\le j\le N\}$), whose coefficients are Lipschitz continuous in the state. Therefore, there is a unique solution $\hat{X}^{1,N}$ of \eqref{1} such that $\e[\sup_{0\le t\le T}|\hat{X}_t^{1,N}|^2]<+\infty$. We refer to \cite{JYXY} for more discussion on SDEs. From \eqref{epsilon_thm_9} and standard estimates for SDEs, we have
	\begin{equation}\label{epsilon_thm_10}
		\begin{split}
			\e[\sup_{0\le t\le T}|X_t^{1}-\hat{X}_t^{1,N}|^2]\le \ &  \frac{C_1}{N}\Big(1+\e[|\xi^1|^2+\int_0^T( |{u}_t|^2+|\bar{u}^{1}_t|^2)dt]\Big) \\
			& + \frac{C_1}{N^2}\e[\int_0^T |{u}_t|^2dt].
		\end{split}
	\end{equation}
	We define
	\begin{equation}\label{2}
		\begin{split}
			&\hat{X}^{1}_t:=\xi^1+\int_0^t \left[b_0(s,m_s)+b_1(s)\hat{X}^1_s+b_2(s){u}_s \right] ds\\
			&\qquad\qquad +\int_0^t \left[ \sigma_0(s,m_s)+\sigma_1(s)\hat{X}^1_s+\sigma_2(s){u}_s \right] dW_s^1\\
			&\qquad\qquad +\int_0^t \left[ \tilde{\sigma}_0(s,m_s)+\tilde{\sigma}_1(s)\hat{X}^1_s+\tilde{\sigma}_2(s){u}_s \right] dW_s^0,\quad t\in[0,T].
		\end{split}		
	\end{equation}
	Similar as \eqref{1}, equation \eqref{2} is a SDE for $\hat{X}^{1}$ (given $u$ and $m$), whose coefficients are Lipschitz continuous in the state. Therefore, there is a unique solution $\hat{X}^{1}$ of \eqref{2} such that $\e[\sup_{0\le t\le T}|\hat{X}_t^{1}|^2]<+\infty$. From \eqref{5_2} and standard estimates for SDEs, we have
	\begin{equation*}\label{epsilon_thm_11}
		\e\Big[\sup_{0\le t\le T}|\hat{X}_t^{1,N}-\hat{X}_t^{1}|^2\Big] \le \frac{C_1}{N}\left(1+\e[|\xi^1|^2+\int_0^T |\bar{u}^{1}_t|^2dt]\right)+ \frac{C(L,T)}{N^2}\e\left[\int_0^T |{u}_t|^2dt\right]. 
	\end{equation*}
	From the last inequality and \eqref{epsilon_thm_10}, we have
	\begin{equation}\label{epsilon_thm_12}
		\begin{split}
			\e\Big[\sup_{0\le t\le T}|{X}_t^{1}-\hat{X}_t^{1}|^2\Big]\le\  & \frac{C_1}{N}\left(1+\e\left[|\xi^1|^2+\int_0^T\left( |{u}_t|^2+|\bar{u}^{1}_t|^2\right)dt\right]\right) \\
			& + \frac{C_1}{N^2}\e\left[\int_0^T |{u}_t|^2dt\right]. 
		\end{split}
	\end{equation}  
	Recall that
	\begin{equation}\label{v_6}
		\begin{split}
			J_1(u|\bar{u}^{-1})=\e\bigg[\int_0^T\frac{1}{N}\sum_{j=1}^N \overline{f}(t,{X}_t^1,u_t,{X}_t^{j},X^0_t)dt+\frac{1}{N}\sum_{j=1}^N \overline{g}({X}_T^1,{X}_T^{j},X_T^0)\bigg].
		\end{split}
	\end{equation}    
	From Assumption (H2), Lemma~\ref{epsilon_lemma}, \eqref{epsilon_thm_9}, \eqref{epsilon_thm_12} and \eqref{boundedness}, similar as \eqref{v_3}, we have for any $2\le j\le N$,   
	\begin{equation}\label{v_7}
		\begin{split}
			&\sum_{j=2}^N\e[\int_0^T\frac{1}{N}|\overline{f}(t,{X}_t^1,u_t,{X}_t^{j},X^0_t)-\overline{f}(t,\hat{X}_t^1,u_t,\bar{X}_t^j,\bar{X}_t^0)|dt]\\
			\le\ & \frac{3L}{N}\sum_{j=2}^N\e\bigg[\int_0^T(1+|{X}_t^1|+ |\hat{X}_t^1|+|u_t|+|{X}_t^{j}|+|\bar{X}_t^{j}|+|X^0_t|+|\bar{X}_t^{0}|)\\
			&\qquad\qquad\qquad \cdot  \left(|{X}_t^1-\hat{X}_t^1|+|{X}_t^{j}-\bar{X}_t^{j}|+|{X}_t^0-\bar{X}_t^0|\right) dt\bigg]\\
			\le\ & C_1\left(1+\e\left[|\xi^0|^2+|\xi^1|^2+\int_0^T \left(|u_t|^2+|\bar{u}_t^{0}|^2+|\bar{u}_t^{1}|^2\right) dt\right]\right)^{\frac{1}{2}}\\
			&\qquad \cdot \Big( \e\Big[\sup_{0\le t\le T}|{X}_t^{1}-\hat{X}_t^{1}|^2+\sup_{0\le t\le T}|{X}_t^{j}-\bar{X}_t^{j}|^2+\sup_{0\le t\le T}|{X}_t^{0}-\bar{X}_t^{0}|^2\Big]\Big)^\frac{1}{2}\\
			\le\ & \frac{C_1}{\sqrt{N}}\left(1+\e\left[|\xi^0|^2+|\xi^1|^2+\int_0^T \left(|u_t|^2+|\bar{u}_t^{0}|^2+|\bar{u}_t^{1}|^2\right) dt\right]\right)\\
			\le\ & \frac{C}{\sqrt{N}} \left(1+\e\left[\int_0^T |u_t|^2 dt\right]\right).
		\end{split}
	\end{equation}
	Here and in the rest of the proof, $C$ always stands for a constant depending only on all the quantities in~\eqref{v_5}. For $j=1$, from Assumption (H2) and \eqref{boundedness}, we have
	\begin{align}
		&\e\left[\int_0^T\frac{1}{N}\left|\overline{f}(t,{X}_t^1,u_t,{X}_t^{1},X^0_t)-\overline{f}(t,\hat{X}_t^1,u_t,\bar{X}_t^1,\bar{X}_t^0)\right|dt\right] \notag\\
		\le\ & \frac{3L}{N}\e\left[\int_0^T\left(1+|{X}_t^1|^2+ |\hat{X}_t^1|^2+|u_t|^2+|{X}_t^{1}|^2+|\bar{X}_t^{1}|^2+|X^0_t|^2+|\bar{X}_t^{0}|^2\right) dt \right] \notag\\
		\le\ & \frac{C_1}{N}\left(1+\e\left[|\xi^0|^2+|\xi^1|^2+\int_0^T \left(|u_t|^2+|\bar{u}_t^{0}|^2+|\bar{u}_t^{1}|^2\right) dt\right]\right) \notag\\
		\le\ & \frac{C}{N} \left(1+\e[\int_0^T |u_t|^2 dt]\right). \label{v_8}
	\end{align}
	Plugging \eqref{v_7} and \eqref{v_8} (and a similar inequality for $\overline{g}$) into \eqref{v_6}, we know that 
	\begin{equation}\label{epsilon_thm_13}
		\begin{split}
			J_1(u|\bar{u}^{-1})\geq\ &\e\bigg[\int_0^T \frac{1}{N}\sum_{j=1}^N \overline{f}(t,\hat{X}_t^1,u_t,\bar{X}_t^j,\bar{X}_t^0)dt+\frac{1}{N}\sum_{j=1}^N \overline{g}(\hat{X}_T^1,\bar{X}_T^j,\bar{X}_T^0)\bigg]\\
			&-\frac{C}{\sqrt{N}}\bigg(1+\e\left[\int_0^T|u_t|^2 dt\right]\bigg).
		\end{split}
	\end{equation} 
	We set $Y_t^j:=\overline{f}(t,\hat{X}_t^1,u_t,\bar{X}_t^j,\bar{X}_t^0)-f(t,\hat{X}_t^1,u_t,m_t,\bar{X}_t^0)$ for $2\le j\le N$, then, from the definition of $m_t$ and the fact that $\f_t^j$ is independent from $\{W^1,\xi^1\}$, we have $\e[Y_t^j|\f_t^1]=0$. For $2\le j\neq j'\le N$, since $(W^1,W^j,W^{j'})$ are independent, we have $\e[Y_t^j Y_t^{j'}|\f_t^1]=\e[Y_t^j|\f_t^1]\e[Y_t^{j'}|\f_t^1]=0$, and then $\e[Y_t^j Y_t^{j'}]=0$. Since $\{Y_t^j,2\le j\le N\}$ are identically distributed, we have 
	\begin{equation}\label{4}
		\begin{split}
			&\e\bigg[\bigg|\frac{1}{N}\sum_{j=1}^N(\overline{f}(t,\hat{X}_t^1,u_t,\bar{X}_t^j,\bar{X}_t^0)-f(t,\hat{X}_t^1,u_t,m_t,\bar{X}_t^0))\bigg|^2\bigg]\\
			=\ &\frac{1}{N^2} \sum_{j=1}^N\e\left[|Y_t^j|^2\right]+\frac{1}{N^2}\sum_{1\le j\neq j'\le N}\e\left[Y_t^jY_t^{j'}\right]\\
			=\ &\frac{1}{N^2} \sum_{j=1}^N\e\left[|Y_t^j|^2\right]+\frac{2}{N^2}\sum_{j=2}^N\e\left[Y_t^1Y_t^{j}\right] \\
			\le\ & \frac{2}{N^2} \sum_{j=2}^N\e\left[|Y_t^j|^2\right]+\frac{1}{N}\e\left[|Y_t^1|^2\right] \\
			\le\ &  \frac{C}{N} \bigg(1+\e\left[\int_0^T|u_t|^2 dt\right]^2\bigg). 
		\end{split}
	\end{equation}
	Similarly, 
	\begin{equation}\label{v_9}
		\e\bigg[\bigg|\frac{1}{N}\sum_{j=1}^N\left(\overline{g}(\hat{X}^1_T,\bar{X}_T^j,\bar{X}_T^0)-g(\hat{X}^0_T,m_T,\bar{X}_T^0)\right)\bigg|^2\bigg] \le \frac{C}{N}  \bigg(1+\e\left[\int_0^T|u_t|^2 dt\right]^2\bigg).
	\end{equation}
	From \eqref{4} and \eqref{v_9}, we know that
	\begin{align}
		&\e\bigg[\int_0^T\frac{1}{N}\sum_{j=1}^N \overline{f}(t,\hat{X}_t^1,u_t,\bar{X}_t^j,\bar{X}_t^0)dt+\frac{1}{N}\sum_{j=1}^N \overline{g}(\hat{X}_T^1,\bar{X}_T^j,\bar{X}_T^0)\bigg] \notag\\
		\geq\ & \e\bigg[\int_0^T f(t,\hat{X}_t^1,u_t,m_t,\bar{X}_t^0)dt+ g(\hat{X}_T^1,m_T,\bar{X}_T^0)\bigg] \notag\\
		& -\frac{C}{\sqrt{N}}\bigg(1+\e\left[\int_0^T|u_t|^2 dt\right]\bigg) \notag\\
		=\ & J_1(u|\bar{u}^0,m)-\frac{C}{\sqrt{N}}\bigg(1+\e\left[\int_0^T|u_t|^2 dt\right]\bigg). \label{epsilon_thm_14}
	\end{align}
	From \eqref{epsilon_thm_13} and \eqref{epsilon_thm_14}, we know that
	\begin{equation}\label{epsilon_thm_15}
		\begin{split}
			J_1(u|\bar{u}^{-1})&\geq J_1(u|\bar{u}^0,m)-\frac{C}{\sqrt{N}}\e\left[\int_0^T|u_t|^2 dt\right]-\frac{C}{\sqrt{N}}.
		\end{split}
	\end{equation}
	In a similar way, we have
	\begin{equation}\label{epsilon_thm_16}
		\begin{split}
			J_1(\bar{u}^1|\bar{u}^{-1})\le J_1(\bar{u}^1|\bar{u}^0,m)+\frac{C}{\sqrt{N}}.
		\end{split}
	\end{equation}
	From Theorem~\ref{thm:smp} we know that	
	\begin{equation}\label{epsilon_thm_17}
		J_1(\bar{u}^1|\bar{u}^0,m)=\inf_{u\in \lr^2_{\f^1}(0,T)}J_1(u|\bar{u}^0,m).
	\end{equation}  
	From the following Lemma~\ref{lem:1}, we know that
	\begin{equation}\label{epsilon_thm_17.1}
		\begin{split}
			&J_1(u|\bar{u}^0,m)-\frac{C}{\sqrt{N}}\e\left[\int_0^T|u_t|^2 dt\right]\\
			\geq\ & \inf_{u\in \lr^2_{\f^1}(0,T)}\left\{J_1(u|\bar{u}^0,m)-\frac{C}{\sqrt{N}}\e\left[\int_0^T|u_t|^2 dt\right]\right\}\\
			\geq\ & \inf_{u\in \lr^2_{\f^1}(0,T)}J_1(u|\bar{u}^0,m)-\frac{C}{\sqrt{N}}.
		\end{split}
	\end{equation}      	
	Therefore, we know from \eqref{epsilon_thm_15}-\eqref{epsilon_thm_17.1} that
	\begin{equation*}\label{epsilon_thm_18}
		J_1(\bar{u}^1|\bar{u}^{-1})-\frac{C}{\sqrt{N}}\le J_1(u|\bar{u}^{-1}).
	\end{equation*}
\end{proof}

To complete the previous proof, it remains to state and prove the following lemma.

\begin{lemma}\label{lem:1}
	Let ${W}=\{W_t,0\le t\le T\}$ be a one-dimensional Brownian motion and $\xi$ be a square integrable random variable. Let $\f$ be the natural filtration of $(\xi,W)$. For a control $u\in\lr_{\f}^2(0,T)$, we consider the dynamical system
	\begin{align*}
		X^u_t:=\xi+\int_0^t b(s,X^u_s,u_s)ds+\int_0^t\sigma(s,X^u_s,u_s)dW_s,
	\end{align*}
	and the cost function
	\begin{align*}
		J(u):=\e\left[\int_0^T f(t,X^u_t,u_t) dt+g(X^u_T)\right],
	\end{align*}
	where coefficients $(b,\sigma,f,g)$ satisfy Assumptions (H1)-(H4) (except that they are independent of $m$ and $x^0$). Then, we have	
	\begin{equation}\label{lem1_0}
		\begin{split}
			\inf_{u\in \lr^2_{\f}(0,T)}\left\{J(u)-\kappa\e\left[\int_0^T|u_t|^2 dt\right]\right\}-\inf_{u\in \lr^2_{\f}(0,T)}J(u)& \le C \left(1+\e[|\xi|^2] \right)\kappa
		\end{split}
	\end{equation}    
	for a positive constant $C$  depending only on $(L,T)$.  
\end{lemma}

\begin{proof}
	For $\kappa\in[0,\frac{C_f}{2}]$, we denote by 
	\begin{align*}
		J^\kappa(u):=J(u)-\kappa\e\left[\int_0^T|u_t|^2 dt\right],\quad u\in \lr_{\f}^2(0,T).
	\end{align*}
	From \cite[Theorem 3.1]{HT} we know that $u^\kappa$ is the unique optimal control for problem $\inf_{u\in \lr^2_{\f}(0,T)}J^\kappa(u)$, where $(X^\kappa,p^\kappa;u^\kappa,q^\kappa)\in (\sr_{\f}^2)^2\times (\lr_{\f}^2)^2(0,T)$ is the unique solution of the following FBSDEs
	\begin{equation}\label{lem1_1}
		\left\{
		\begin{aligned}
			&dX^\kappa_t=b(t,X^\kappa_t,u^\kappa_t)dt+\sigma(t,X^\kappa_t,u^\kappa_t)dW_t,\quad t\in(0,T];\\
			&dp^\kappa_t= -\left[b_1(t)p^\kappa_t+\sigma_1(t)q^\kappa_t+f_x(t,X^\kappa_t,u^\kappa_t)\right] dt+q^\kappa_tdW_t,\quad t\in[0,T);\\
			&X^\kappa_0=\xi,\quad p^\kappa_T=g_x(X^\kappa_T),
		\end{aligned}
		\right.   
	\end{equation}
	with the optimal condition
	\begin{align}\label{lem1_2}
		b_2(t)p_t^\kappa+\sigma_2(t)q_t^\kappa+f_u(t,X^\kappa_t,u^\kappa_t)-2\kappa u^\kappa_t=0,\quad t\in[0,T].
	\end{align}
	The existence and uniqueness of solutions of FBSDEs \eqref{lem1_1}-\eqref{lem1_2} is an immediate consequence of \cite[Theorem 2.3]{SP}. From Assumptions (H1)-(H2) and standard estimates for SDEs and BSDEs, we have
	\begin{equation}\label{lem1_6}
		\begin{split}
			&\e\left[\sup_{0\le t\le T}|( X^\kappa_t, p^\kappa_t)|^2+\int_0^T | q^\kappa_t|^2 dt\right]\le C\e\left[1+|\xi|^2+\int_0^T  | u^\kappa_t|^2 dt\right].
		\end{split}
	\end{equation}
	Here and in the rest of the proof, $C$ always stands for a constant depending only on $(L,T)$, which may vary from line to line. Applying It\^o's lemma on $p^\kappa_t X^\kappa_t$ and taking expectation, we have
	\begin{equation}\label{lem1_3}
		\begin{split}
			&\e\left[p^\kappa_T X^\kappa_T\right]-\e\left[ p^\kappa_0\xi\right]\\
			=\ & \e\left[\int_0^T [b_0(t) p^\kappa_t+\sigma_0(t) q^\kappa_t-f_{x}(t,X_t^\kappa,u_t^\kappa)X^\kappa_t-f_{u}(t,X_t^\kappa,u_t^\kappa)u^\kappa_t+2\kappa|u_t^\kappa|^2 ]dt\right].
		\end{split}
	\end{equation}
	From Assumption (H3), we have
	\begin{equation}\label{lem1_4}
		\begin{split}
			&f_{x}(t,X_t^\kappa,u_t^\kappa)X^\kappa_t+f_{u}(t,X_t^\kappa,u_t^\kappa)u^\kappa_t \geq f_{x}(t,0,0)X^\kappa_t+f_{u}(t,0,0)u^\kappa_t+2C_f|u_t^\kappa|^2,\\
			&g_x(X^\kappa_T) X^\kappa_T \geq g_x(0) X^\kappa_T.
		\end{split}
	\end{equation}
	Plugging \eqref{lem1_4} into \eqref{lem1_3}, from Assumption (H2), we have for $\epsilon\in(0,1)$,
	\begin{equation}\label{lem1_5}
		\begin{split}
			&2(C_f-\kappa)\e\left[\int_0^T |u_t^\kappa|^2 dt\right]\\
			\le\ & \e\left[p^\kappa_0\xi-g_x(0) X^\kappa_T+\int_0^T \left[b_0(t) p^\kappa_t+\sigma_0(t) q^\kappa_t-f_{x}(t,0,0)X^\kappa_t-f_{u}(t,0,0)u^\kappa_t \right]dt\right]\\
			\le\ & \frac{C}{\epsilon}\left(1+\e\left[|\xi|^2\right]\right)+\epsilon\e\bigg[\sup_{0\le t\le T}|(X^\kappa_t,p^\kappa_t)|^2+\int_0^T |(u^\kappa_t,q^\kappa_t)|^2dt\bigg].
		\end{split}
	\end{equation}	
	Plugging \eqref{lem1_6} into \eqref{lem1_5} and choosing $\epsilon$ small enough, from the fact that $\kappa\le \frac{C_f}{2}$ and $C_f\geq L^{-1}$, we have 
	\begin{equation}\label{lem1_7}
		\begin{split}
			&\e\left[\int_0^T |u_t^\kappa|^2 dt\right]\le C\left(1+\e\left[|\xi|^2\right]\right).
		\end{split}
	\end{equation}
	Plugging \eqref{lem1_7} back into \eqref{lem1_6}, we have
	\begin{equation}\label{lem1_8}
		\e\bigg[\sup_{0\le t\le T}|(X^\kappa_t,p^\kappa_t)|^2+\int_0^T |(u^\kappa_t,q^\kappa_t)|^2dt\bigg]\le C\left(1+\e\left[|\xi|^2\right]\right).
	\end{equation}
	That is, the norm of $(X^\kappa,u^\kappa,p^\kappa,q^\kappa)$ is bounded by a constant independent of $\kappa$ for $\kappa\in[0,\frac{C_f}{2}]$. We denote by $(\hat{X},\hat{u},\hat{p},\hat{q})$ the solution of FBSDEs \eqref{lem1_1}-\eqref{lem1_2} when $\kappa=0$, and denote by 
	\begin{align*}
		(\de X^\kappa,\de u^\kappa,\de p^\kappa,\de q^\kappa):=(X^\kappa-\hat{X},u^\kappa-\hat{u},p^\kappa-\hat{p},q^\kappa-\hat{q}).
	\end{align*}
	Then, $(\de X^\kappa,\de u^\kappa,\de p^\kappa,\de q^\kappa)$ satisfy the following FBSDEs
	\begin{equation*}
		\left\{
		\begin{aligned}
			&d\de X^\kappa_t=\left[b_1(t)\de X^\kappa_t+b_2(t)\de u^\kappa_t\right]dt+\left[\sigma_1(t)\de X^\kappa_t+\sigma_2(t)\de u^\kappa_t\right]dW_t,\quad t\in(0,T];\\
			&d\de p^\kappa_t= -[b_1(t)\de p^\kappa_t+\sigma_1(t)\de q^\kappa_t+f_x(t,X^\kappa_t,u^\kappa_t)-f_x(t,\hat{X}_t,\hat{u}_t)] dt\\
			&\qquad\qquad +\de q^\kappa_tdW_t,\quad t\in[0,T);\\
			&\de X^\kappa_0=0,\quad \de p^\kappa_T=g_x(X^\kappa_T)-g_x(\hat{X}_T),
		\end{aligned}
		\right.   
	\end{equation*}
	with the optimal condition
	\begin{align}\label{lem1_10}
		b_2(t)\de p_t^\kappa+\sigma_2(t)\de q_t^\kappa+f_u(t,X^\kappa_t,u^\kappa_t)-f_u(t,\hat{X}_t,\hat{u}_t)-2\kappa u^\kappa_t=0,\quad t\in[0,T].
	\end{align}
	From Assumptions (H1)-(H2) and standard estimates for SDEs and BSDEs, we have
	\begin{equation}\label{lem1_14}
		\begin{split}
			&\e\left[\sup_{0\le t\le T}|(\de X^\kappa_t, \de p^\kappa_t)|^2+\int_0^T | \de q^\kappa_t|^2 dt\right]\le C\e\left[\int_0^T  |\de u^\kappa_t|^2 dt\right].
		\end{split}
	\end{equation}
	Applying It\^o's lemma on $\de p^\kappa_t \de X^\kappa_t$ and taking expectation, from \eqref{lem1_10}, we have
	\begin{equation}\label{lem1_9}
		\begin{split}
			&\e\left[\de p^\kappa_T \de X^\kappa_T\right]\\
			=\ & \e\bigg[\int_0^T 2\kappa u_t^\kappa \de u^\kappa_t-\big[(f_{x},f_u)(t,X_t^\kappa,u_t^\kappa)-(f_x,f_u)(t,\hat{X}_t,\hat{u}_t)\big]\cdot(\de X^\kappa_t,\de u^\kappa_t) dt\bigg].
		\end{split}
	\end{equation}
	From Assumption (H3), we have
	\begin{equation}\label{lem1_11}
		\begin{split}
			&\big[(f_{x},f_u)(t,X_t^\kappa,u_t^\kappa)-(f_x,f_u)(t,\hat{X}_t,\hat{u}_t)\big]\cdot(\de X^\kappa_t,\de u^\kappa_t) \geq 2C_f|\de u_t^\kappa|^2,\\
			&\big(g_x(X^\kappa_T)-g_x(\hat{X}_T)\big) \de X^\kappa_T \geq 0.
		\end{split}
	\end{equation}
	Plugging \eqref{lem1_11} into \eqref{lem1_9}, from Cauchy's inequality, we have
	\begin{equation}\label{lem1_12}
		\begin{split}
			C_f \e\left[\int_0^T |\de u_t^\kappa|^2 dt\right]&\le \kappa \e\left[\int_0^T  |u_t^\kappa| |\de u^\kappa_t| dt\right]\\
			&\le \kappa \bigg(\e\left[\int_0^T |u_t^\kappa|^2 dt\right]\bigg)^{\frac{1}{2}}\bigg(\e\left[\int_0^T |\de u_t^\kappa|^2 dt\right]\bigg)^\frac{1}{2}.
		\end{split}
	\end{equation}
	Plugging \eqref{lem1_7} into \eqref{lem1_12}, we have
	\begin{equation}\label{lem1_13}
		\begin{split}
			\e\left[\int_0^T |\de u_t^\kappa|^2 dt\right]&\le \kappa^2 C\left(1+\e\left[|\xi|^2\right]\right).
		\end{split}
	\end{equation}
	Plugging \eqref{lem1_13} back into \eqref{lem1_14}, we have
	\begin{equation}\label{lem1_15}
		\begin{split}
			&\e\left[\sup_{0\le t\le T}\left|(\de X^\kappa_t, \de p^\kappa_t)\right|^2+\int_0^T \left|(\de u^\kappa_t, \de q^\kappa_t)\right|^2 dt\right]\le \kappa^2 C\left(1+\e\left[|\xi|^2\right]\right).
		\end{split}
	\end{equation}
	From Assumption (H2), \eqref{lem1_7}, \eqref{lem1_8}, and \eqref{lem1_13}, we have that
	\begin{align*}
		&|J^\kappa(u^\kappa)-J(\hat{u})| \\
		\le\ & |J(u^\kappa)-J(\hat{u})|+\kappa\e\left[\int_0^T|u^\kappa_t|^2 dt\right]\\
		=\ & \e\left[\int_0^T \left( f(t,X^\kappa_t,u^\kappa_t)-f(t,\hat{X}_t,\hat{u}_t)\right) dt+g(X^\kappa_T)-g(\hat{X}_T)+\kappa \int_0^T|u^\kappa_t|^2 dt\right]\\
		\le\ & 2L\e\bigg[\int_0^T \left(1+|X^\kappa_t|+|\hat{X}_t|+|u^\kappa_t|+|\hat{u}_t|\right)\left(|\de X^\kappa|+|\de u^\kappa|\right) dt \\
		&\qquad + \left(1+|X^\kappa_T|+|\hat{X}_T|\right) |\de X^\kappa|+\kappa \int_0^T|u^\kappa_t|^2 dt\bigg]\\
		\le\ & C\left(1+\e\left[|\xi|^2\right] \right) \kappa,
	\end{align*}
	which yields~\eqref{lem1_0}. 
\end{proof}

%%%%%%%%%%%%%%%%%%%%%%%%%%%%%%%%%%%%%%%%%%%%%%%%%%%%%%%%%%%%%%%%%%%%%%%%%%

%%%%%%%%%%%%%%%%%%%%%%%%%%%%%%%%%%%%%%%%%%%%%%%%%%%%%%%%%%%%%%%%%%%%

%%%%%%%%%%%%%%%%%%%%%%%%%%%%%%%%%%%%%%%%%%%%%%


\footnotesize
\begin{thebibliography}{99}
	\bibitem{SA}S. Ahuja, \emph{Wellposedness of mean field games with common noise under a weak monotonicity condition}, SIAM J. Control Optim. {54.1} (2016), pp. 30--48.
	
	\bibitem{SAWR}S. Ahuja, W. Ren, and  T. W. Yang, \emph{Forward-backward stochastic differential equations with monotone functionals and mean field games with common noise}, {Stoch. Proc. Appl.} (2019).
	
	\bibitem{common_3}E. Bayraktar, A. Cecchin, A. Cohen and F. Delarue, \emph{Finite state mean field games with Wright-Risher common noise}, J. Math. Pures Appl. 147 (2021), pp. 98--162.
	
	\bibitem{AB2}A. Bensoussan, J. Frehse and S. C. P. Yam, \emph{The Master equation in mean field theory}, J. Math. Pure. Appl. {103.6} (2015), pp. 1441--1474.
	
	\bibitem{PA}R. Carmona and  F. Delarue, \emph{Probabilistic analysis of mean-field games}, {SIAM J. Control Optim.} {51.4} (2012), pp. 2705--2734.
	
	\bibitem{book_mfg}R. Carmona and F. Delarue, \emph{Probabilistic theory of mean field games with applications II}, Probability Theory and Stochastic Modelling, 84, Springer, Cham, 2018.
	
	\bibitem{common_1}R. Carmona, F. Delarue and D. Lacker, \emph{Probabilistic analysis of mean field games with a common noise}, Ann. Probab. 44 (2016), pp. 3740--3803.
	
	\bibitem{CRXZ}R. Carmona and X. Zhu, \emph{A probabilistic approach to mean field games with major and minor players}, {Ann. Appl. Probab.} {26.3} (2016), pp. 1535--1580.
	
	\bibitem{common_2}A. Cecchin and F. Delarue, \emph{Selection by vanishing common noise for potential finite state mean field games}, Commun. Partial Differential Equations 47 (2022), pp. 89--168.
	
	\bibitem{GW}W. Gangbo, A. R. M\'esz\'aros, C. Mou and J. Zhang,  \emph{Mean field games master equations with nonseparable Hamiltonians and displacement monotonicity}, Ann. Probab. 50.6 (2022), 2178--2217.
	
	\bibitem{HFC}F. Hu, \emph{Stochastic LQ games involving a major player and a large number of minor players}, {J. Fudan Univ. Nat. Sci.} {50.6} (2011), pp. 720--731.
	
	\bibitem{YH2}Y. Hu and S. Peng, \emph{Solution of forward-backward stochastic differential equations}, {Probab.Theory Rel.} {103.2} (1995), pp. 273--283.
	
	\bibitem{MH2}M. Huang, \emph{Large-population LQG games involving a major player: The Nash certainty equilvanece principle}, {SIAM J. Control Optim.} {48.5} (2010), pp. 3318--3353.
	
	\bibitem{MH1}M. Huang, R. P. Malhamé, and P.E. Caines, \emph{Large population stochastic dynamic games: closed-loop McKean–Vlasov systems and the Nash certainty equivalence principle}, {Commun. Inf. Syst.} {6.3} (2006), pp. 221--251.
	
	\bibitem{HT}Z. Huang and S. Tang, \emph{Mean field games with common noises and conditional distribution dependent FBSDEs}, {Chinese Ann. Math. B} {43.4} (2022), pp. 523--548.
	
	\bibitem{VN}V. N. Kolokoltsov, J. Li, and W. Yang, \emph{Mean field games and nonlinear Markov processes}, preprint (2012). Available at arXiv, 1112.3744v2.
	
	\bibitem{NM}M. Nourian and  P. E. Caines, \emph{$\epsilon$-nash mean field game theory for nonlinear stochastic dynamical systems with major and minor agents}, {SIAM J. Control Optim.} {51.4} (2012), pp. 2090--2095.
	
	\bibitem{SP}S. Peng and  Z. Wu, \emph{Fully Coupled Forward-backward stochastic differential equations and applications to optimal control}, {SIAM J. Control Optim.} {\bf 37.3} (1999), pp. 825--843.	
	
	\bibitem{HP}H. Pham, \emph{Continuous-time stochastic control and optimization with financial applications}, Springer, 2009.
	
	\bibitem{AS}A.-S. Sznitman, \emph{Topics in propagation of chaos}, {École d’Été de Probabilités de Saint-Flour XIX--1989, eds: Lecture Notes in Math.}, Springer, Berlin {1464} (1991), 165--251.
	
	\bibitem{JYXY}J. Yong and X. Y. Zhou, \emph{Stochastic controls: Hamiltonian systems and HJB equations}, Springer, 1999.
\end{thebibliography}
\end{document}